\documentclass[twoside, 11pt, reqno]{amsart}

\usepackage{amsfonts}
\usepackage[utf8]{inputenc}
\usepackage{array}

\usepackage[english]{babel}
\usepackage{latexsym}
\usepackage{amscd}
\usepackage{amsmath}
\usepackage{amssymb}
\usepackage{amsthm}
\usepackage{graphicx}
\usepackage{mathrsfs}
\usepackage{mathtools}
\usepackage{pict2e}
\usepackage{comment}
\usepackage{stackrel} 
\usepackage{wasysym}
\usepackage[all]{xy}
\usepackage[dvipsnames]{xcolor}
\usepackage[colorlinks,final,hyperindex]{hyperref}
\usepackage[noabbrev,capitalize]{cleveref}
\usepackage{tikz}
\usepackage{tikz-cd}
\usetikzlibrary{decorations.pathmorphing,decorations.markings,arrows,calc,shapes.geometric,arrows.meta,positioning}

\usepackage{geometry}
\usepackage{enumitem}

\setlist[enumerate]{font=\normalfont}

\allowdisplaybreaks

\newtheorem*{introtransmission}{Theorem~\ref{thm formality transmission}}
\newtheorem*{introinfinity}{Corollary~\ref{appli B}}
\newtheorem*{intropuncturing}{Corollary~\ref{formality under puncturing}}
\newtheorem*{introintrinsic}{Theorem~\ref{intrinsic}}
\newtheorem*{introchange}{Corollaries~\ref{change en char p} and~\ref{zinbiel en char p}}

\newtheorem*{intropddomination}{Corollary~\ref{cor:pd-domination}}

\newtheorem{lemma}{Lemma}[section]
\newtheorem{theorem}[lemma]{Theorem}

\newtheorem{corollary}[lemma]{Corollary}
\newtheorem{proposition}[lemma]{Proposition}

\newtheorem*{notation}{Notation}
\numberwithin{equation}{section}

\theoremstyle{definition}
\newtheorem{definition}[lemma]{Definition}
\newtheorem{remark}[lemma]{\sc Remark}
\newtheorem{example}[lemma]{\sc Example}

\newtheorem{assumption}[lemma]{\sc Assumptions}

\newcommand{\nocontentsline}[3]{}
\newcommand{\tocless}[2]{\bgroup\let\addcontentsline=\nocontentsline#1{#2}\egroup}

\hypersetup{citecolor=BleuMinuit, linkcolor=bordeau, filecolor=black, urlcolor=BleuMinuit}

\title{Formality transmission}
\author{Coline Emprin}
\address{Coline Emprin, Department of Mathematics, Stockholm University, SE-106 91, Sweden}
\email{\noindent \href{email:coline.emprin@math.su.se}{coline.emprin@math.su.se}}

\date{\today}


\usepackage{etoc}

\newcommand{\sommairesection}{%
	\begingroup
	\etocsettocstyle{}{}
	\etocsetnexttocdepth{subsection}
	
	\etocsetstyle{subsection}
	{}
	{}
	{  \etocifnumbered{\etocnumber\quad}{}%
		\etocname\par}
	{}
	
	\localtableofcontents
	\endgroup
}
\begin{document}
\definecolor{red}{RGB}{230,97,0}
\definecolor{blue}{RGB}{93,58,155}
\definecolor{Chocolat}{rgb}{0.36, 0.2, 0.09}
\definecolor{BleuTresFonce}{rgb}{0.215, 0.215, 0.36}
\definecolor{BleuMinuit}{RGB}{0, 51, 102}
\definecolor{bordeau}{rgb}{0.5,0,0}
\definecolor{BORDEAU}{rgb}{0.5,0,0}
\definecolor{turquoise}{RGB}{6, 62, 62}
\definecolor{rose}{RGB}{235, 62, 124}

\newcommand{\coline}[1]{\textcolor{rose}{#1}}

\newcommand{\hooklongrightarrow}{\lhook\joinrel\longrightarrow}


\newcommand{\chdiscr}{\mathsf{discr.\ Ch}}
\newcommand{\sSe}{\mathsf{sSet}}
\newcommand{\sLialg}{\ensuremath{\mathcal{sL}_\infty\text{-}\,\mathsf{alg}}}
\newcommand{\isLialg}{\ensuremath{\infty\text{-}\,\mathcal{sL}_\infty\text{-}\,\mathsf{alg}}}
\newcommand{\sLiealg}{\ensuremath{\mathcal{s}\lie\,\text{-}\,\mathsf{alg}}}
\def\cD{\Delta}

\newcommand{\ucomnalg}{\ensuremath{\mathrm{uCom}_{\leslant 0}\text{-}\,\mathsf{alg}}}
\newcommand{\fnQsp}{\ensuremath{\mathsf{ft}\mathbb{Q}\text{-}\mathsf{ho}(\mathsf{sSet})}}
\newcommand{\fnsp}{\ensuremath{\mathsf{fn}\text{-}\mathsf{Sp}}}
\newcommand{\nQsp}{\ensuremath{\mathsf{n}\mathbb{Q}\text{-}\mathsf{Sp}}}
\newcommand{\nsp}{\ensuremath{\mathsf{n}\text{-}\mathsf{Sp}}}
\newcommand{\fQalg}{\ensuremath{\mathsf{ft}\text{-}\mathsf{ho}(\mathrm{uCom}_{\leqslant 0}\text{-}\,\mathsf{alg})}}

\newcommand{\coker}{\operatorname{coker}}
\newcommand{\Ran}{\operatorname{Ran}}
\newcommand{\ho}{\operatorname{ho}}
\newcommand{\Mal}{\operatorname{Mal}}
\newcommand{\Loc}{\operatorname{Loc}}


\newcommand{\sLi}{\ensuremath{\mathcal{sL}_\infty}}
\newcommand{\ccoLi}{\ensuremath{\mathrm{cBCom}}}
\newcommand{\sLie}{\ensuremath{\mathcal{s}\mathrm{Lie}}}
\newcommand{\com}{\ensuremath{\mathrm{Com}}}
\newcommand{\ucom}{\ensuremath{\mathrm{uCom}}}
\newcommand{\Cobar}{\ensuremath{\Omega}}
\newcommand{\hatCobar}{\ensuremath{\widehat{\Omega}}}
\renewcommand{\Bar}{\ensuremath{\mathrm{B}}}
\newcommand{\hatBar}{\ensuremath{\widehat{\mathrm{B}}}}
\newcommand{\lie}{\ensuremath{\mathrm{Lie}}}
\newcommand{\uhocom}{\ensuremath{\mathrm{u}\Omega\mathrm{BCom}}}

\newcommand{\Sh}{\ensuremath{\mathrm{Sh}}}


\newcommand{\g}{\ensuremath{\mathfrak{g}}}
\newcommand{\wsLi}{\ensuremath{\widehat{\mathcal{sL}_\infty}}}


\newcommand{\Rh}{\ensuremath{\mathrm{R^h}}}
\renewcommand{\L}{\ensuremath{\mathscr{L}}}
\newcommand{\Li}{\mathfrak{L}}
\newcommand{\APL}{\mathrm{A_{PL}}}
\newcommand{\CPL}{\mathrm{C_{PL}}}


\newcommand{\antishriek}{\text{\raisebox{\depth}{\textexclamdown}}}
\newcommand{\RT}{\mathrm{RT}}
\newcommand{\LRT}{\mathrm{LRT}}
\newcommand{\Sy}{\mathbb{S}}
\renewcommand{\d}{\ensuremath{\mathrm{d}}}
\newcommand{\PP}{\ensuremath{\mathrm{P}}}
\def\Ho#1#2{\Lambda^{#2}_{#1}}
\def\De#1{\Delta^{#1}}
\newcommand{\NN}{\mathbb{N}}
\newcommand{\RR}{\mathbb{R}}
\def\BCH{\mathrm{BCH}}
\newcommand{\wPT}{\ensuremath{\mathrm{wPT}}}
\newcommand{\oPT}{\ensuremath{\overline{\mathrm{PT}}}}
\newcommand{\PaRT}{\ensuremath{\mathrm{PaRT}}}
\newcommand{\PaPT}{\ensuremath{\mathrm{PaPT}}}
\newcommand{\PaPRT}{\ensuremath{\mathrm{PaPRT}}}
\def\rmC{\mathrm{C}}
\newcommand{\berglund}{\ensuremath{\mathcal{B}}}
\def\hot{\widehat{\otimes}} 
\def\whk{\widehat{k}}

\def\colim{\mathop{\mathrm{colim}}}

\newcommand{\CC}{\ensuremath{\mathrm{CC}_\infty}}


\newcommand{\Lalg}{\ensuremath{\mathscr{L}_\infty\text{-}\mathsf{alg}}}

\newcommand{\F}{\ensuremath{\mathrm{F}}}
\newcommand{\h}{\ensuremath{\mathfrak{h}}}

\newcommand{\QQ}{\ensuremath{\mathbb{Q}}}
\newcommand{\D}{\ensuremath{\mathscr{D}}}
\renewcommand{\P}{\ensuremath{\mathcal{P}}}
\newcommand{\C}{\ensuremath{\mathcal{C}}}
\newcommand{\VdL}{\ensuremath{\mathrm{VdL}}}
\newcommand{\ch}{\ensuremath{\mathrm{Ch}}}
\newcommand{\End}{\ensuremath{\mathrm{End}}}
\newcommand{\Aut}{\ensuremath{\mathrm{Aut}}}
\newcommand{\eend}{\ensuremath{\mathrm{end}}}
\newcommand{\susp}{\ensuremath{\mathscr{S}}}
\newcommand{\T}{\ensuremath{\mathcal{T}}}
\newcommand{\vdl}{\ensuremath{\mathrm{VdL}}}
\newcommand{\Tw}{\ensuremath{\mathrm{Tw}}}
\newcommand{\Hom}{\ensuremath{\mathrm{Hom}}}
\renewcommand{\S}{\ensuremath{\mathbb{S}}}
\renewcommand{\k}{\ensuremath{\mathbb{K}}}
\newcommand{\id}{\ensuremath{\mathrm{id}}}
\newcommand{\MC}{\ensuremath{\mathrm{MC}}}
\newcommand{\mc}{\ensuremath{\mathfrak{mc}}}
\newcommand{\mclie}{\ensuremath{\overline{\mathfrak{mc}}}}
\newcommand{\dgl}{\ensuremath{\mathsf{dgLie}}}

\newcommand{\B}{\mathcal{B}}
\newcommand{\Z}{\mathbb{Z}}

\newcommand{\ad}{\operatorname{ad}}
\newcommand{\PTt}{\ensuremath{\widetilde{\mathrm{PT}}}}
\newcommand{\PT}{\ensuremath{\mathrm{PT}}}

\newcommand{\A}{\ensuremath{\mathrm{A}}}

\makeatletter

\newcommand{\R}{\mathbb{R}}
\newcommand{\K}{\mathcal{K}}
\newcommand{\N}{\mathbb{N}}
\newcommand{\calA}{\mathcal{A}}
\newcommand{\calB}{\mathcal{B}}
\newcommand{\calC}{\mathcal{C}}
\newcommand{\calD}{\mathcal{D}}
\newcommand{\calE}{\mathcal{E}}
\newcommand{\calF}{\mathcal{F}}
\newcommand{\calG}{\mathcal{G}}
\newcommand{\calH}{\mathcal{H}}
\newcommand{\calI}{\mathcal{I}}
\newcommand{\calJ}{\mathcal{J}}
\newcommand{\calK}{\mathcal{K}}
\newcommand{\calL}{\mathcal{L}}
\newcommand{\calM}{\mathcal{M}}
\newcommand{\calN}{\mathcal{N}}
\newcommand{\calO}{\mathcal{O}}
\newcommand{\calP}{\mathcal{P}}
\newcommand{\calQ}{\mathcal{Q}}
\newcommand{\calR}{\mathcal{R}}
\newcommand{\calS}{\mathcal{S}}
\newcommand{\calT}{\mathcal{T}}
\newcommand{\calU}{\mathcal{U}}
\newcommand{\calV}{\mathcal{V}}
\newcommand{\calW}{\mathcal{W}}
\newcommand{\calX}{\mathcal{X}}
\newcommand{\calY}{\mathcal{Y}}
\newcommand{\calZ}{\mathcal{Z}}
\newcommand{\kk}{{\Bbbk}}

\newcommand{\ev}{\mathop{\mathrm{ev}}\nolimits}
\newcommand{\co}{\mathop{\mathrm{co}}\nolimits}
\newcommand{\pt}{\mathop{\mathrm{pt}}\nolimits}
\newcommand{\wt}{\mathop{\mathrm{wt}}\nolimits}
\newcommand{\pCY}{\mathop{\mathrm{pCY}}\nolimits}
\newcommand{\gr}{\mathop{\mathrm{gr}}\nolimits}

\newcommand{\Imm}{\mathop{\mathrm{Im}}\nolimits}

\tikzset{>={Stealth[scale=1.2]}}
\tikzset{->-/.style={decoration={
			markings,
			mark=at position #1 with {\arrow{>}}},postaction={decorate}}}
\tikzset{-w-/.style={decoration={
			markings,
			mark=at position #1 with {\arrow{Stealth[fill=white,scale=1.4]}}},postaction={decorate}}}
\tikzset{->-/.default=0.65}
\tikzset{-w-/.default=0.65}
\tikzstyle{bullet}=[circle,fill=black,inner sep=0.5mm]
\tikzstyle{circ}=[circle,draw=black,fill=white,inner sep=0.5mm]
\tikzstyle{vertex}=[circle,draw=black,thick,inner sep=0.5mm]
\tikzstyle{dot}=[draw,circle,fill=black,minimum size=0.5mm,inner sep = 0mm, outer sep = 0mm]
\tikzset{darrow/.style={double distance = 4pt,>={Implies},->},
	darrowthin/.style={double equal sign distance,>={Implies},->},
	tarrow/.style={-,preaction={draw,darrow}},
	qarrow/.style={preaction={draw,darrow,shorten >=0pt},shorten >=1pt,-,double,double
		distance=0.2pt}}
\newcommand\tikcirc[1][2.5]{\tikz[baseline=-#1]{\draw[thick](0,0)circle[radius=#1mm];}}
\newcommand{\tikzfig}[1]{\begin{tikzpicture}[auto,baseline={([yshift=-.5ex]current bounding box.center)}]#1\end{tikzpicture}}
\usetikzlibrary{decorations.pathreplacing}

\newcommand{\DD}{\EuScript D}
\newcommand{\BB}{\EuScript B}
\newcommand{\OO}{\EuScript O}
\newcommand{\KK}{\EuScript K}
\newcommand{\YY}{\EuScript Y}
\newcommand\nn         {\nonumber \\}
\newcommand{\pf}{{\it Proof.}\hspace{2ex}}
\newcommand{\apf}{{\it Another Proof.}\hspace{2ex}}
\newcommand{\epf}{\hspace*{\fill}\emph{$\square$}}
\newcommand{\halmos}{\rule{1ex}{1.4ex}}
\newcommand{\pfbox}{\hspace*{\fill}\emph{$\halmos$}}

\newcommand{\ac}{\mathop{ac}\nolimits}
\newcommand{\sC}{\mathop{s^{-1}\overline{C}}\nolimits}
\newcommand{\sA}{\mathop{s\overline{A}}\nolimits}
\newcommand{\Rep}{\mathop{\mathrm{Rep}}\nolimits}
\newcommand{\Span}{\mathop{\mathrm{Span}}\nolimits}
\newcommand{\proj}{\mathop{\mathrm{proj}}\nolimits}
\newcommand{\PHom}{\mathop{\mathrm{PHom}}\nolimits}
\newcommand{\Dim}{\mathop{\mathrm{Dim}}\nolimits}

\newcommand{\Aus}{\mathop{\mathrm{Aus}}\nolimits}
\newcommand{\Ann}{\mathop{\mathrm{Ann}}\nolimits}
\newcommand{\APC}{\mathop{\mathrm{APC}}\nolimits}
\newcommand{\Barr}{\mathop{\mathrm{Bar}}\nolimits}
\newcommand{\Cone}{\mathop{\mathrm{Cone}}\nolimits}
\newcommand{\modu}{\mathop{\mathrm{mod}}\nolimits}
\newcommand{\DGMod}{\mathop{\mathrm{DGMod}}\nolimits}
\newcommand{\soc}{\mathop{\mathrm{soc}}\nolimits}
\newcommand{\dg}{\mathop{\mathrm{dg}}\nolimits}
\newcommand{\red}{\mathop{\mathrm{red}}\nolimits}
\newcommand{\Map}{\mathop{\mathrm{Map}}\nolimits}
\newcommand{\inj}{\mathop{\mathrm{inj}}\nolimits}
\newcommand{\op}{\mathop{\mathrm{op}}\nolimits}
\newcommand\M          {{\mathcal{M}}}
\newcommand{\Ker}{\mathop{\mathrm{Ker}}\nolimits}
\newcommand{\Tor}{\mathop{\mathrm{Tor}}\nolimits}
\newcommand{\Tot}{\mathop{\mathrm{Tot}}\nolimits}
\newcommand{\Ext}{\mathop{\mathrm{Ext}}\nolimits}
\newcommand{\sg}{\mathop{\mathrm{sg}}\nolimits}
\newcommand{\Sl}{\mathop{\mathfrak{sl}}\nolimits}
\newcommand{\nc}{\mathop{\mathrm{nc}}\nolimits}
\newcommand{\Tr}{\mathop{\mathrm{Tr}}\nolimits}
\newcommand{\Irr}{\mathop{\mathrm{Irr}}\nolimits}
\newcommand{\HC}{\mathop{\mathrm{HC}}\nolimits}
\newcommand{\HH}{\mathop{\mathrm{HH}}\nolimits}
\newcommand{\THH}{\mathop{\mathrm{TH}}\nolimits}
\newcommand{\Perf}{\mathop{\mathrm{Perf}}\nolimits}
\newcommand{\del}{\partial}

\newcommand{\verteq}{\rotatebox{90}{$\;\;=\;\;$}}

\begin{abstract}
Across mathematics, formality results are usually established through general criteria rather than by constructing explicit equivalences. These criteria include domination, puncturing,   highly-connected formality, intrinsic formality, formality descent,  formality in families, purity, etc. In this article, we prove that all these criteria are instances of a single theorem,
\emph{formality transmission}. This unifying framework allows us
to extend these criteria in several directions, beyond characteristic zero and to differential graded algebras over operads or properads, leading to diverse new formality results.
\end{abstract}

\maketitle
	
\makeatletter
\def\@tocline#1#2#3#4#5#6#7{\relax
	\ifnum #1>\c@tocdepth 
	\else
	\par \addpenalty\@secpenalty\addvspace{#2}%
	\begingroup \hyphenpenalty\@M
	\@ifempty{#4}{%
		\@tempdima\csname r@tocindent\number#1\endcsname\relax
	}{%
		\@tempdima#4\relax
	}%
	\parindent\z@ \leftskip#3\relax \advance\leftskip\@tempdima\relax
	\rightskip\@pnumwidth plus4em \parfillskip-\@pnumwidth
	#5\leavevmode\hskip-\@tempdima
	\ifcase #1
	\or\or \hskip 1em \or \hskip 2em \else \hskip 3em \fi%
	#6\nobreak\relax
	\hfill\hbox to\@pnumwidth{\@tocpagenum{#7}}\par
	\nobreak
	\endgroup
	\fi}

\newcommand{\enableopenany}{%
	\@openrightfalse%
}
\makeatother
	
	\setcounter{tocdepth}{1}
	\etockeeporiginaltableofcontents
	\tableofcontents

\section*{\textcolor{bordeau}{Introduction}}

 Let $R$ be a commutative ring and let $\P$ denote a type of
algebraic structure, such as an associative algebra, a Frobenius algebra,
or a pre-Calabi--Yau algebra. A differential graded (dg) $\P$-algebra
$(A,\varphi)$ is \emph{formal} if it can be connected to its cohomology,
equipped with the induced structure, by a zig-zag of quasi-isomorphisms:
\[
(A,\varphi)
\overset{\sim}{\longleftarrow}\cdot
\overset{\sim}{\longrightarrow}\cdots
\overset{\sim}{\longleftarrow}\cdot
\overset{\sim}{\longrightarrow}(H(A),\varphi_*).
\]
Here a quasi-isomorphism is a morphism inducing an isomorphism in
cohomology. Formality plays a role in a broad range of mathematical
problems, including:
\begin{enumerate}
	\item purity in algebraic geometry~\cite{Deligne80};
	\item deformation quantization of Poisson manifolds in mathematical
	physics~\cite{Kon97,Hinich03bis};
	\item Koszul duality for the category $\mathcal O$ of a reductive group
	in representation theory~\cite{RSW14};
	\item the existence of a non-Sasakian $K$-contact manifold in contact
	geometry~\cite{BFVT16};
	\item splitting the de Rham complex of Drinfeld space in the $p$-adic
	Langlands program~\cite{BQ24}.
\end{enumerate} The notion originated in rational homotopy theory. A space $X$ is called \emph{formal over $\mathbb Q$} if its rational singular cochains are
formal as an associative dg algebra. For a connected nilpotent space of
finite type, this means that its rational homotopy type, and
hence its invariants, can be recovered from its cohomology algebra.
More generally, we call $X$ \emph{formal over $R$} if
$C^*_{\mathrm{sing}}(X;R)$ is formal as an associative dg algebra.  

\medskip

\noindent A central result is the rational formality of compact K\"ahler manifolds, proved by Deligne, Griffiths, Morgan, and Sullivan~\cite{DGMS75}. Their proof uses Hodge theory to construct an explicit zig-zag of
quasi-isomorphisms. This is a distinctive
feature: later formality results were often established without directly constructing such a zig-zag. An alternative approach is to appeal to formality criteria which include:  domination, puncturing and high connectivity, intrinsic formality, descent and families, lifts of homology automorphisms, and changes of algebraic structure.

\noindent These criteria share a common principle: the formality of a structure is deduced from that of another suitably related structure. In this article, we make this principle precise by showing that all these criteria are instances of a single theorem, \emph{formality transmission}. As explained below, this framework allows us to extend these results to other coefficient rings and to a broader range of algebraic structures. Hence, one can ask for the formality of spaces,  algebras encoded by operads, but also operads themselves, structures involving operations with  several outputs such as dg pre-Calabi-Yau algebras etc. Such structures are encoded by generalizations of operads: colored operads and properads. 

\bigskip 

\paragraph{\bf Hypotheses} Let $A$ be a chain complex over  $R$. Let $\C$ be a reduced weight-graded dg coproperad (resp. colored dg cooperad), and let  $(A, \varphi)$ be an $\Omega \C$-algebra structure.

\bigskip

\noindent Formality transmission rests on a deformation-theoretic approach to formality. Under suitable hypotheses on $R$  and $\C$, homotopy transfer yields an $\Omega\C$-algebra structure on $H(A)$,
\[
\varphi_t=\varphi_*+\varphi^{(2)}_t+\varphi^{(3)}_t+\cdots
\]
 together with $\infty$-quasi-isomorphisms relating it to
$(A,\varphi)$, see~\cite{LodayVallette12,berglund,PHC,LRL23} and the references therein. The transferred structure extends the induced structure
$\varphi_*$. Formality is thus governed by whether its higher
operations can be removed. An algebra admitting such a transferred structure is said to be
\begin{itemize}
	\item[$\centerdot$] \emph{gauge formal} if $\varphi_t$ and $\varphi_*$
	are related by an $\infty$-isomorphism;
	\item[$\centerdot$] \emph{gauge $n$-formal}, for $n\geqslant1$, if
	$\varphi_t$ is related by an $\infty$-isomorphism to a structure
	\[
	\psi=\varphi_*+\psi^{(n+2)}+\psi^{(n+3)}+\cdots.
	\]
\end{itemize}
Over a characteristic zero field, formality and gauge formality
coincide. Over a general ring, the comparison requires additional
hypotheses, see \cite[Proposition~2.18]{CE24b}. Algebraic structures on the homology are Maurer--Cartan elements of the
convolution algebra $\mathfrak g_{H(A)}$, and their deformations are
governed by its twisted complex~\cite{graphexp}:
\[
D_A
=\bigl(\mathfrak g_{H(A)},d+[\varphi_t,-]\bigr).
\]
This recovers the Hochschild, Harrison, and
Chevalley--Eilenberg deformation complexes in the associative,
commutative, and Lie cases, respectively. We write $E(\C,H(A),\varphi_t)_r^{p,q}$  for the spectral sequence of $D_A$, where $p$ is the weight,
and $q$ is the total homological degree. In this framework, Kaledin classes~\cite{Kal07,Lun07,MR19,Kaledin} \[
K_\varphi 
\in H_{-1}\!\left(\mathfrak g_{H(A)}[[\hbar]]^{\varphi_t}\right) 
\] provide a complete
obstruction to gauge formality under the stated coefficient conditions. Over a $\mathbb Q$-algebra, $K_\varphi=0$ if and only if
$(A,\varphi)$ is gauge formal. More generally, its truncation $K_\varphi^n$ modulo $\hbar^n$ vanishes if and only if the algebra is gauge $n$-formal,
provided that $n!\in R^\times$.
Formality transmission then rests on the following detection principle: a Kaledin class vanishes if its image under a filtered comparison vanishes and the maps on the relevant spectral-sequence terms are injective.

\begin{introtransmission}[Formality transmission]
Suppose that $(A,\varphi)$ admits a transferred
	structure and fix $	f\colon D_A\longrightarrow(\mathfrak h,d')$ a filtered chain map to a weight-graded complex. For $k\geqslant2$, set
	\[
	\theta_k:=E_k^{k,-1}(f)\colon
	E_k^{k,-1}(D_A)\longrightarrow E_k^{k,-1}(\mathfrak h).
	\]
	\begin{enumerate}
		\item If $R$ is a $\mathbb Q$-algebra, $f_*(K_\varphi)=0$, and
		$\theta_k$ is injective for every $k\geqslant2$, then $(A,\varphi)$
		is gauge formal.
		\item Let $n\geqslant1$ with $n!\in R^\times$. If
		$f_*(K_\varphi^n)=0$ and $\theta_k$ is injective for
		$2\leqslant k\leqslant n+1$, then $(A,\varphi)$ is gauge $n$-formal.
	\end{enumerate}
\end{introtransmission}

\noindent A direct corollary addresses the following question:   let
$f\colon(A,\varphi)\rightsquigarrow(B,\psi)$ be an $\infty$-morphism,
when does formality of one algebra imply that of the other? This leads to the following criterion involving the properadic spectral sequences $E(\C, A, B, f)$ introduced in Section~\ref{section 2}. For the following corollary, we assume the hypotheses of
Theorem~\ref{homotopy invariance}, so that the spectral sequences of the original
algebras identify with those of their transferred structures.

\begin{introinfinity}
	Let $f\colon(A,\varphi)\rightsquigarrow(B,\psi)$ be an
	$\infty$-morphism of $\Omega\C$-algebras admitting transferred
	structures. For $k\geqslant2$, consider the induced maps
	\[
	E(\C,A,\varphi)_k^{k,-1}
	\xrightarrow{\ \theta_k\ }
	E(\C,A,B,f)_k^{k,-1}
	\xleftarrow{\ \widetilde\theta_k\ }
	E(\C,B,\psi)_k^{k,-1} \ .
	\] 	\begin{enumerate}
		\item If $R$ is a $\mathbb{Q}$-algebra and $\theta_k$ is injective
		for every $k\geqslant 2$, then $(A,\varphi)$ is gauge formal
		whenever $(B,\psi)$ is.	
		
		\item Let $n \geqslant 1$ with $n!$ invertible in $R$.
		If $\theta_k$ is injective for $2\leqslant k\leqslant n+1$,
		then $(A,\varphi)$ is gauge $n$-formal
		whenever $(B,\psi)$ is.	
	\end{enumerate}
	Both statements remain valid with $(A,\varphi)$ and $(B,\psi)$
	interchanged and $\theta_k$ replaced by $\tilde{\theta}_k$.
\end{introinfinity}

\noindent This provides a corrected version of previous formality transfer results, see Remark \ref{error}.  We now explain how the formality criteria mentioned above are special cases of formality transmission and the generalizations this provides.

\medskip 

\paragraph{\bf Formality under domination.}
By \cite{MSZ23}, rational formality of spaces is
preserved under domination: if a map $Y\to X$ from a formal manifold hits the fundamental class of a closed orientable manifold $X$, then $X$ is formal over $\mathbb{Q}$.  Formality transmission extends this criterion to algebras encoded by operads over a commutative ring: an $\infty$-morphism admitting a homotopy retract as an $A$-module induces injective maps $\theta_k$, see Theorem~\ref{formality under domination}.  Applying this to cochains, with the retraction given by
Poincar\'e duality, gives the following corollary.

\begin{intropddomination}	
	Let $f\colon Y\to X$ be a map between spaces homotopy equivalent to finite CW complexes, with finitely generated projective
	$R$-cohomology. Suppose that $X$ is an $R$-oriented
	Poincar\'e duality space of dimension $d$ and that
	\[
	f_*c=[X]
	\qquad\text{for some }c\in H_d(Y;R).
	\]
	If $C^*(Y;R)$ is gauge $n$-formal, so is $C^*(X;R)$ provided $n ! \in R^{\times}$. 
\end{intropddomination}

\noindent This corollary leads to the following restriction on mapping degrees. Let $X$ and $Y$ be equidimensional closed connected oriented manifolds and let $\ell$ be prime. For $1\leqslant n<\ell$, if $Y$ is gauge $n$-formal over $\mathbb{F}_\ell$ but $X$ is
not, $\ell$ divides the degree of every map $Y\to X$. Among other examples, the same method applies to finite-group invariants, see Corollary~\ref{cor:finite-invariants}. For example, \cite{DCH21} established partial formality of hyperplane arrangement complements. Transmission along finite coverings extends their bounds to quotients by free finite group actions of invertible order, see Corollary \ref{cor:arrangement-quotients}.

\paragraph{\bf Formality under puncturing.}
Stasheff proved that a closed simply connected manifold is
rationally formal if and only if its puncture is~\cite{Sta83}.
Formality transmission leads to generalizations to other coefficient
rings, to Poincar\'e duality complexes and to \'etale cochains.

\begin{intropuncturing}
	Let $M^d$ be a closed connected $R$-oriented manifold,
	$d\geqslant2$, with $H^*(M;R)$ finite projective and
	$H^1(M;R)=0$. Let us fix $x \in M$. 
	If $d+i\in R^\times$ for $1\leqslant i\leqslant d-4$ then  \[
	C^*(M;R)\text{ gauge formal}
	\quad\Longleftrightarrow\quad
C^*(M\setminus\{x\};R)\text{ gauge formal}.
	\] The invertibility condition is automatic when $d\leqslant4$ or when $R$ is a $\mathbb Q$-algebra.
\end{intropuncturing}

\noindent
For example, if $M$ is simply connected and $R = \mathbb F_p$ formality under puncturing holds when $p>2d-4$. Corollary~\ref{cor:etale-puncturing} is the \'etale version under analogous 
cohomological assumptions.

\bigskip

\paragraph{\bf Intrinsic formality.}
A graded $\P$-algebra $H$ is \emph{intrinsically formal} if every
dg $\P$-algebra inducing $H$ in cohomology is gauge formal. Diverse criteria like
suitable vanishing conditions on its deformation cohomology ensure intrinsic formality~\cite{Sul77,HJ79,FH82,MP91,Hinich03bis,Lup90,Kaledin}.
This key property is at the heart
of the above problem (2). Specializing formality transmission to the case $B = 0$ yields the following intrinsic formality criteria, generalizing former criteria by removing all the characteristic restrictions from \cite{Kaledin}.

\begin{introintrinsic}
	Assume that the $\Omega\C$-algebras under consideration admit
	transferred structures. Let $(H,\psi)$ be a strict $\Omega\C$-algebra  and set
	$E=E(\C,H,\psi)$.
	\begin{enumerate}
		\item If $E_2^{k,-1}=0$ for every $k\geqslant2$, then $(H,\psi)$
		is intrinsically formal.
		\item Let $n\geqslant1$. If $E_2^{k,-1}=0$ for
		$2\leqslant k\leqslant n+1$, then $(H,\psi)$ is intrinsically
		$n$-formal.
	\end{enumerate}
\end{introintrinsic}

\noindent An application of this criterion in the properadic setting is given by the joint work with Alex Takeda \cite{ET24}, which establishes intrinsic coformality for all spheres in characteristic zero.

\bigskip

\paragraph{\bf Formality descent}
Let $S$ be a faithfully flat commutative $R$-algebra. Under transfer and finite-presentation
hypotheses, $A$ is gauge $n$-formal if and only if
$A\otimes_R S$ is, provided $n!\in R^\times$, see \cite{Kaledin}. Over a $\mathbb Q$-algebra, the same equivalence holds
for gauge formality, see
\cite{Sul77,NM78,HJ79,GNPR05,Kaledin}.
Applying transmission to the extension
$A\to A\otimes_R S$ recovers this descent theorem, see Theorem~\ref{formality descent}.
In this setting, the coefficient assumptions remain unchanged. Proposition~\ref{failure} shows why they cannot be removed in general: full gauge formality need not descend under faithfully flat extensions in positive characteristic.

\bigskip 

\paragraph{\bf Formality in families} Under suitable flatness, finiteness, and torsion-freeness hypotheses, formality of the generic fibre of a family over an integral $\mathbb{Q}$-algebra implies formality of every fibre
\cite{Kal07,Lun07,BB20, Kaledin}.
Theorem~\ref{formality in families} revisits the corresponding
criteria of~\cite[Section~4.2]{Kaledin} as a consequence of formality transmission. As an arithmetic application, Corollary \ref{cor:el-haouari} recovers El Haouari's
criterion~\cite{El92}.

\bigskip

\paragraph{\bf Lifting homology automorphisms.}
The existence of chain-level lifts of suitable homology automorphisms
is another source of formality criteria~\cite[Section~4.4]{Kaledin}.
The classical example is the grading automorphism $\sigma_\alpha$,
which acts by multiplication by $\alpha^k$ in degree $k$, where
$\alpha$ is a unit and $\alpha^k-1\in R^\times$ for every $k>0$.
 This leads to the \emph{purity criterion}
\cite{Sul77,GNPR05,Petersen14,Dup16,CH20,DCH21,CS22,CH22} at the heart of problems (1), (3) and (5) above: a chain-level lift of $\sigma_\alpha$ then implies
gauge formality. This approach is extended by \cite[Theorem~4.9]{Kaledin}  to more general homology automorphisms. Formality transmission refines this result in two directions:  relaxing the invertibility hypothesis therein by injectivity at the spectral sequence level
	(Theorem~\ref{relevements}) and under stronger assumptions on the automorphism, the conclusion holds over arbitrary
	commutative rings, without inverting factorials
	(Theorem~\ref{rev2}).

\medskip

\paragraph{\bf Change of algebraic structure.}
Over a field of characteristic zero,
a commutative dg algebra is formal if and only if its
underlying associative dg algebra is formal~\cite{Sal17}.   An operadic version of this result is given in \cite{CPW}. Let $f\colon\mathcal P\to\mathcal Q$ be a morphism of binary Koszul operads over a characteristic-zero field, with both
operads finite-dimensional in each arity.
If $f$ admits a retraction as an infinitesimal
$\mathcal P$-bimodule, a dg $\mathcal P^!$-algebra
is formal if and only if its restriction along
$f^!$ is formal.  Using formality transmission, Theorem~\ref{cooperad change} generalizes this criterion requiring injectivity on the
spectral sequence terms instead of a retraction. It provides results over more general coefficient rings, including the following two examples.

\begin{introchange}
	Let $A$ be a commutative dg algebra or a non-unital dg
	Zinbiel algebra over a commutative ring $R$, admitting
	a transferred structure. Restriction along
$	\mathrm{Ass}\to \mathrm{Com}$ or $	\mathrm{Dend}\to \mathrm{Zinb},$
	respectively, preserves and reflects:
	\begin{enumerate}
		\item gauge formality if $R$ is a $\mathbb Q$-algebra;
		\item gauge $n$-formality if $n\geqslant1$ and
		$(n+1)!\in R^\times$.
	\end{enumerate}
\end{introchange}

\noindent
In particular, the finite-order comparison holds over
$\mathbb F_\ell$ for $1\leqslant n\leqslant\ell-2$.

\bigskip

\paragraph{\bf Outline of the paper}
\begin{itemize}
	\item[$\centerdot$] Section~\ref{section 1} works in the general setting of weight-graded dg Lie algebras. We recall deformation theory in Section~\ref{1.1} and the construction of Kaledin classes in Section~\ref{1.2}, before establishing a transmission criterion for gauge triviality in Section~\ref{1.3}. Applied to convolution dg Lie algebras governing algebraic structures, this criterion yields formality transmission in characteristic zero.
	
	\medskip 
	\item[$\centerdot$]  Section~\ref{section 2} develops formality transmission for algebraic structures over arbitrary base rings. We first review properadic calculus (Section~\ref{2.1}), gauge formality (Section~\ref{2.2}), and the adjoint action of $\infty$-morphisms (Section~\ref{2.3}). Section~\ref{2.4} introduces the properadic spectral sequence associated with an $\infty$-morphism and studies its homotopy invariance, leading to the formality transmission criterion in Section~\ref{2.5}.  Section~\ref{Euler classes} is an interlude in which we introduce properadic Euler classes and use them to characterise formality in terms of degeneration of the associated spectral sequence. Section~\ref{sec:colored-setting} discusses the extension of these results to coloured operads.
	
	\medskip
	\item[$\centerdot$]  Finally, Section~\ref{section 4} recovers the criteria discussed above as special cases of formality transmission and generalises them in various settings.
\end{itemize}

\bigskip

\paragraph{\bf Acknowledgments} 
I would like to thank Alexander Berglund, Geoffroy Horel, Marco Manetti and
Bruno Vallette for the fruitful conversations. AI suggested corrections to the proof
of Theorem \ref{thm:pd-transmission} and the counterexample of Proposition \ref{failure}. This work was supported by the Göran Gustafsson Foundation for Research in Natural Sciences and Medicine through Alexander Berglund's 2024 Göran Gustafsson Prize. 

\bigskip

\section{\textcolor{bordeau}{Transmission of gauge triviality}} \label{section 1}

The purpose of this section is to establish a criterion for transmitting gauge triviality. It relies on Kaledin classes which faithfully detect whether a Maurer--Cartan element is gauge trivial, i.e. gauge equivalent to its first component. We first recall the deformation theory governed by dg Lie
algebras in Section~\ref{1.1}, and the construction of Kaledin classes in
Section~\ref{1.2}. This allows us to prove the transmission criterion in
Section~\ref{1.3}. A key application is the study of formality of algebraic structures, as discussed in Section \ref{section 2}. Throughout this section, the ground ring $R$ is a $\mathbb{Q}$-algebra. 

\subsection{Deformations governed by dg Lie algebras}\label{1.1}

 We briefly review the deformation theory governed by dg Lie algebras. We refer the reader to  \cite{DSV22} for more details.

\begin{definition}[dg Lie algebra and Maurer--Cartan elements]
	A \emph{dg Lie algebra} is a triple
\[\left(\mathfrak{g},[-,-],d\right)\] consisting of a graded $R$-module $\mathfrak{g}$ together with

\begin{itemize}
	\item[$\centerdot$] a Lie bracket $[-,-]\colon \mathfrak{g}\otimes \mathfrak{g}\to \mathfrak{g}$, namely a homogeneous map of degree $0$ satisfying graded anti-symmetry and the graded Jacobi identity;
	\item[$\centerdot$] a differential $d\colon \mathfrak{g}\to \mathfrak{g}$, that is, a degree $-1$ derivation such that $d^2=0$.
\end{itemize} A dg Lie algebra $\mathfrak{g}$ is called \emph{complete} if it is equipped with a decreasing filtration by subcomplexes compatible with the Lie bracket, 
	\[
	\mathfrak{g}=\mathcal{F}^1\mathfrak{g}
	\supset
	\mathcal{F}^2\mathfrak{g}
	\supset
	\mathcal{F}^3\mathfrak{g}
	\supset \cdots,
	 \qquad   \qquad
	\left[\mathcal{F}^n\mathfrak{g},\mathcal{F}^m\mathfrak{g}\right]
	\subset
	\mathcal{F}^{n+m}\mathfrak{g},
	\]
	and such that $\mathfrak{g}$ is recovered as the inverse limit of its quotients:
	\[
	\mathfrak{g}
	\xrightarrow{\ \cong\ }
	\lim_{n\in\mathbb{N}}
	\mathfrak{g}/\mathcal{F}^n\mathfrak{g}.
	\]	An element $\varphi\in \mathfrak{g}_{-1}$ is called a \emph{Maurer--Cartan element} if it satisfies
	\[
	d(\varphi)+\tfrac{1}{2}[\varphi,\varphi]=0.
	\]
	The set of Maurer--Cartan elements of $\mathfrak{g}$ will be denoted by $\mathrm{MC}(\mathfrak{g}) \ .$
In what follows, all dg Lie algebras are assumed to be complete without further mention. The notation $\mathfrak{g}$ will be used both for the underlying chain complex and for the complete dg Lie algebra structure.
\end{definition}

\begin{definition}[Gauge action and equivalences]
The Baker--Campbell--Hausdorff formula defines a group structure on the degree zero part $\mathfrak{g}_0$ of any dg Lie algebra $\mathfrak{g}$. This group \[\left(\mathfrak{g}_0,\mathrm{BCH},0\right)\] is called the \emph{gauge group} of $\mathfrak{g}$. It acts on the set of Maurer--Cartan elements through the \emph{gauge action} as follows. For $\lambda\in \mathfrak{g}_0$ and $\varphi\in \mathrm{MC}(\mathfrak{g})$, we set 
	\[
	\lambda\cdot \varphi
	\coloneqq
	e^{\mathrm{ad}_{\lambda}}(\varphi)
	-
	\frac{e^{\mathrm{ad}_{\lambda}}-\mathrm{id}}
	{\mathrm{ad}_{\lambda}}(d\lambda) \ ,
	\] where $\mathrm{ad}_{\lambda}(\varphi) \coloneqq [\lambda, \varphi]$. Two Maurer--Cartan elements $\varphi$ and $\psi$ are said to be \emph{gauge equivalent} if there exists $\lambda\in \mathfrak{g}_0$ such that $\lambda\cdot \varphi=\psi \ .$
\end{definition}

\begin{proposition}\label{technique}	Let $\mathfrak{g}$ be a dg Lie algebra and let $\varphi \in \mathrm{MC}(\mathfrak{g})$ be a Maurer--Cartan element. 
	\begin{enumerate}
		\item The operator $d^{\varphi}\coloneqq d+\mathrm{ad}_{\varphi}$ is a differential and defines a \emph{twisted dg Lie} algebra $\mathfrak{g}^{\varphi}$.
	\item  	For every $\lambda \in \mathfrak{g}_0$, the morphism  $e^{\mathrm{ad}_{\lambda}} : \mathfrak{g}^{\varphi}\to \mathfrak{g}^{\lambda \cdot \varphi}$ is an isomorphism of dg Lie algebras whose inverse is given by $e^{\mathrm{ad}_{-\lambda}}$. 
	\end{enumerate}

\end{proposition}

\begin{proof}
	See \cite[Proposition~1.43]{DSV22} and \cite[Lemma 1.9]{Kaledin}.
\end{proof}

\subsection{Kaledin classes in dg Lie algebras}\label{1.2}

Kaledin classes, introduced in \cite[Section~1]{Kaledin}, faithfully detect whether a given Maurer--Cartan element is gauge trivial, i.e. gauge equivalent to its lowest-weight component. The purpose of this section is to briefly recall their construction.  Let $R[\![\hbar]\!]$ be the ring of formal power series and let $\mathfrak{g}$ be any dg Lie algebra. We consider the dg Lie algebra of formal power series \[\mathfrak{g}[\![\hbar]\!] := \mathfrak{g} \ \widehat{\otimes} \ R[\![\hbar]\!] \ .\]

\begin{definition}[Formal deformation]
	Let $\varphi_0\in \mathrm{MC}(\mathfrak{g})$ be a Maurer--Cartan element. A \emph{formal deformation} of $\varphi_0$ is a Maurer--Cartan element
	\[
	\Phi
	=
	\varphi_0+\varphi_1\hbar+\varphi_2\hbar^2+\cdots
	\in
	\mathrm{MC}\bigl(\mathfrak{g}[\![\hbar]\!]\bigr)
	\]
	whose reduction modulo $(\hbar)$ is  $\varphi_0$. We denote the set of such deformations by
	\[
	\mathrm{Def}_{\varphi_0}(R[\![\hbar]\!]) \ .
	\]	For any graded $R$-module $V$~, the \emph{differentiation operator} $\partial_{\hbar} : V[\![\hbar]\!] \longrightarrow V[\![\hbar]\!]$ is defined by \[\sum_{i \geqslant0} x_i \hbar^i \longmapsto \sum_{i \geqslant1} i x_i \hbar^{i-1} \ .\]
\end{definition}

\begin{definition}[Kaledin classes]
	Let $\varphi_0\in \mathrm{MC}(\mathfrak{g})$, and let $\Phi$ be a formal deformation of $\varphi_0$.
	
	\begin{itemize}
		\item[$\centerdot$]  The \emph{Kaledin class} of $\Phi$ is the homology class
		\[
		K_{\Phi}
		\coloneqq
		\left[\partial_{\hbar}\Phi\right]
		\in
		H_{-1}\left(\mathfrak{g}[\![\hbar]\!]^{\Phi}\right).
		\]
		 
		\item[$\centerdot$] For every $n\geqslant 1$, the \emph{$n^{\mathrm{th}}$ truncated Kaledin class} of $\Phi$ is the homology class
		\[
		K_{\Phi}^n
		\coloneqq
		\left[
		\varphi_1+2\varphi_2\hbar+\cdots+n\varphi_n\hbar^{n-1}
		\right]
		\in
		H_{-1}\left(
		\left(\mathfrak{g}[\![\hbar]\!]/(\hbar^n)\right)^{\overline{\Phi}^{\,n}}
		\right) 
		\] where we denote by $\overline{\Phi}^{\,n}$ the image of $\Phi$ in
		$	\mathfrak{g}[\![\hbar]\!]/(\hbar^n).$
	\end{itemize}
\end{definition}

\noindent Kaledin classes behave naturally with respect to gauge transformations.

\begin{lemma}[{\cite[Lemmas~1.19 and 1.22]{Kaledin}}]\label{iso}
	Let $\Phi\in\mathrm{Def}_{\varphi_0}(R[\![\hbar]\!])$. The isomorphism
	\[
	e^{\mathrm{ad}_{\lambda}}\colon
	\mathfrak{g}[\![\hbar]\!]^{\Phi}
	\longrightarrow
	\mathfrak{g}[\![\hbar]\!]^{\lambda\cdot\Phi} \ , \quad \lambda\in\hbar\,\mathfrak{g}_0[\![\hbar]\!]
	\]
	sends $K_{\Phi}$ to $K_{\lambda\cdot\Phi}$ in homology. The same compatibility holds
	modulo $\hbar^n$ for every $n\geqslant 1$.
\end{lemma}

\begin{definition}[weight-graded dg Lie algebra and gauge triviality]
A \emph{weight-graded dg Lie algebra} is a dg Lie algebra together with an additional weight grading such that \[\mathfrak{g} \cong \prod_{k \geqslant 1} \mathfrak{g}^{(k)}, \quad d \left(\mathfrak{g}^{(k)} \right) \subset \mathfrak{g}^{(k + 1)}, \quad \left[ \mathfrak{g}^{(k)},  \mathfrak{g}^{(l)} \right] \subset  \mathfrak{g}^{(k +l)}  \ .  \]   Every element $\varphi$ in $\mathfrak{g}$ decomposes as $\varphi = \varphi^{(1)} + \varphi^{(2)} +  \varphi^{(3)} + \cdots $~, where $\varphi^{(k)} \in \mathfrak{g}^{(k)}$~, for all $k \geqslant 1$~. It is complete for the descending filtration defined by \[\mathcal{F}^n \mathfrak{g} \coloneqq  \prod_{k \geqslant n} \mathfrak{g}^{(k)} \ . \] A  Maurer--Cartan element $\varphi \in \mathrm{MC}\left(\mathfrak{g}\right)$ is

\begin{enumerate}
	\item \emph{gauge trivial} if  there exists $\lambda \in \mathfrak{g}_0$ such that $\lambda \cdot \varphi = \varphi^{(1)}$~.
	\item \emph{gauge $n$-trivial} for $n \geqslant 1$ if there exists $\lambda \in \mathfrak{g}_0 $ such that \[(\lambda \cdot \varphi)^{(k)} = 0 \quad \mbox{for all} \quad 2 \leqslant  k \leqslant   n +1 \ .\]
\end{enumerate}
\end{definition}

\begin{definition}[Prismatic decomposition]
	Let \(V \) be a weight-graded chain complex. The \emph{$s$-prismatic decomposition} of \(v \in  V\)  is the formal power series
	\[
	\mathcal D_s(v)
	:= v^{(0)} + \cdots + 
	v^{(s)}
	+ v^{(s+1)}\hbar
	+ v^{(s+2)}\hbar^2
	+\cdots
	\in V[\![\hbar]\!]
	\] where \(v^{(r)}\) denotes the component of weight \(r\).
\end{definition}

\begin{definition}[Kaledin classes]
	Let $\mathfrak{g}$ be a weight-graded dg Lie algebra and let $\varphi$ be a Maurer--Cartan element. Let us consider its prismatic decomposition $\mathfrak{D}_{1}(\varphi)$, which is a formal deformation of $\varphi^{(1)}$, see \cite[Lemma~1.29]{Kaledin}. 
	\begin{itemize}
		\item[$\centerdot$] The \emph{Kaledin class} $K_{\varphi}$ is the Kaledin class $K_{\mathfrak{D}_{1}(\varphi) }$ of its prismatic decomposition.
		\item[$\centerdot$]
		Its \emph{$n^{\text{th}}$-truncated Kaledin class} $K_{\varphi}^n$ is the $n^{\text{th}}$-truncated Kaledin class $K_{\mathfrak{D}_{1}(\varphi) }^n$.
	\end{itemize}
\end{definition}

\begin{theorem}[{\cite[Theorem 1.31 and Proposition 1.33]{Kaledin}}]\label{Kaledin wg}
	Let $R$ be a $\mathbb{Q}$-algebra. Let $\mathfrak{g}$ be a weight-graded dg Lie algebra over $R$ and let $\varphi $ be a Maurer--Cartan element. 
	\begin{enumerate}
		\item The Kaledin class $K_{\varphi}$ is zero if and only if $\varphi$ is gauge trivial. 
	\item  For $n \geqslant 1$, the truncated Kaledin class $K_{\varphi}^n$ is zero if and only if $\varphi$ is gauge $n$-trivial.  	
	\item The Kaledin class $K_{\Phi}$ is zero if and only if $K^n_{\Phi}$ is zero for every $n \geqslant 1$. 
	\end{enumerate}
\end{theorem}

\subsection{Transmission of gauge triviality}\label{1.3}

We now establish the transmission criterion, Theorem~\ref{thm gauge transmission 2}.
The guiding idea is that, although the vanishing of a Kaledin class may be
difficult to check directly, it can sometimes be detected after applying a
suitable filtered morphism of complexes. Throughout this subsection, a complete weight-graded chain complex is understood to have a differential strictly increasing the weight filtration. Let $(V,d)$ be a complete weight-graded chain complex with filtration
$\mathcal{F}^{p}V=\prod_{j\geqslant p}V^{(j)}$. Write $d^{(r)}$ for the component of $d$ raising weight by $r$.
We index the homological spectral sequence by weight and
total homological degree:
\[
E_0^{p,q}(V)=V_q^{(p)},
\qquad
d_r:E_r^{p,q}(V)\longrightarrow E_r^{p+r,q-1}(V).
\]
We denote by $H_q^{(p)}(V,d^{(1)})$ the homology of
$(V,d^{(1)})$ in degree $q$ and weight $p$.

\begin{lemma}\label{lem spectral-page-identification}
	Let $(V,d)$ be a complete weight-graded chain complex. Suppose that $d^{(r)}=0$ for $1<r<k+1$. Then, for every $(p,q)$, there are
	canonical isomorphisms
	\[
	H_q^{(p)}(V,d^{(1)})
	\cong E_{2}^{p,q}(V)
	\cong\cdots\cong
	E_{k+1}^{p,q}(V).
	\]
\end{lemma}

\begin{proof}
	The differential on $E_1$ is induced by $d^{(1)}$, so
	\[
	E_2^{p,q}(V)\cong H_q^{(p)}(V,d^{(1)}).
	\]
	Since $d$ has no components of weight shift $2,\ldots,k$, the
	differentials $d_2,\ldots,d_k$ vanish. Hence the pages
	$E_2,\ldots,E_{k+1}$ are canonically identified.
\end{proof}

\begin{lemma}
Given  $f\colon (V,d)\to(W,d')$, a morphism of weight-graded chain complexes,
	\[
	\mathcal D_0(f)\colon
	\bigl(V[\![\hbar]\!],\mathcal D_1(d)\bigr)
	\longrightarrow
	\bigl(W[\![\hbar]\!],\mathcal D_1(d')\bigr)
	\]
	is a morphism of chain complexes preserving the \(\hbar\)-adic filtration.
\end{lemma}

\begin{proof} This follows from taking the components of weight 
	\(2+n\), and \(1+n\) in the identities	\[ d^2= d'^2=0 \quad \mbox{and}  \quad fd=d'f \ ,\]  and multiplying them by
	\(\hbar^n\).   The preservation of the
	\(\hbar\)-adic filtration is immediate.
\end{proof}

\begin{theorem}[Gauge triviality transmission]
	\label{thm gauge transmission 2}
	Let $\mathfrak g$ be a weight-graded dg Lie algebra and let
	$\varphi\in\mathrm{MC}(\mathfrak g)$ be a Maurer--Cartan element.  Consider  a filtered morphism of chain complexes
	\[
	f\colon\mathfrak g^\varphi\longrightarrow(\mathfrak h,d') \ ,
	\] where $(\mathfrak h,d')$ is a weight-graded chain complex. Let us set $F\coloneqq\mathcal D_0(f),$ and for $k\geqslant1$,
	\[
	\theta_k\coloneqq E_k^{k,-1}(f)\colon
	E_k^{k,-1}(\mathfrak g^\varphi)
	\longrightarrow
	E_k^{k,-1}(\mathfrak h) \ .
	\]
	\begin{enumerate}
		\item If $F_*(K_\varphi)=0$ and $\theta_k$ is injective for every
		$k>1$, then $\varphi$ is gauge trivial.
		
		\item  If	$F_*(K_\varphi^n)=0$ for $n\geqslant 1$, and $\theta_k$ is injective for every
		$1<k\leqslant n+1$,  then $\varphi$ is gauge $n$-trivial.
	\end{enumerate}
\end{theorem}

\begin{proof}
	We first prove Point~(2). Suppose that $\varphi$ is not gauge $n$-trivial. By Theorem \ref{Kaledin wg} (2) $K_\varphi^n\neq0$. Let
	$m\leqslant n$ be minimal such that $K_\varphi^m\neq0$. By
	Theorem~\ref{Kaledin wg}~(2), since $K_\varphi^{m-1}=0$ there exists
	$\lambda\in\mathfrak g_0$ such that, for
	$\psi\coloneqq\lambda\cdot\varphi$,
	\[
	\psi^{(2)}
	=\cdots=
	\psi^{(m)}
	=0,
	\qquad
	K_\psi^m
	=
	\bigl[m\psi^{(m+1)}\hbar^{m-1}\bigr] \ , 
	\]  taking $\lambda=0$ if $m=1$.  Let us set $	g\coloneqq e^{\operatorname{ad}_{-\lambda}}$ and $G\coloneqq\mathcal D_0(g)$. By
	Lemma~\ref{iso}, we have $G_*(K_\psi^m)=K_\varphi^m.$
	Moreover, reducing the hypothesis modulo $\hbar^m$ gives
	$F_*(K_\varphi^m)=0$ and
	\[
	f^{(0)}
	\left(m\psi^{(m+1)}\right)\hbar^{m-1}
	=
	\mathcal D_1\left(d'\right)(\nu)
	\pmod{\hbar^m},
	\]  for some
	$\nu\in \mathfrak h[\![\hbar]\!]$
	since $g^{(0)}=\mathrm{id}$. Let us write
	\[
	\nu=\sum_{i=0}^{m-1}\nu_i\hbar^i,
	\qquad
	y\coloneqq\sum_{i=0}^{m-1}\nu_i^{(i+1)}.
	\]
	Comparing the coefficient of $\hbar^i$ in weight $i+2$ gives
	\[
	d'y\in\mathcal F^{m+1}\mathfrak h,
	\qquad
	d'y\equiv
	f^{(0)}\bigl(m\psi^{(m+1)}\bigr)
	\pmod{\mathcal F^{m+2}\mathfrak h}.
	\]
	Thus, we have that $f^{(0)}\left(m\psi^{(m+1)}\right)$ is trivial in  $E_{m+1}^{m+1,-1}(\mathfrak h)$. Consequently,
	\[
	\theta_{m+1}\circ E_{m+1}(g)
	\left(\left[m\psi^{(m+1)}\right]\right)=0.
	\]
	Since $E_{m+1}(g)$ is an isomorphism and
	$\theta_{m+1}$ is injective,
	\[
	\left[m\psi^{(m+1)}\right]=0
	\quad\text{in}\quad
	H_{-1}^{(m+1)}
	\left(\mathfrak g,d^{\psi^{(1)}}\right).
	\]
	It follows that $K_\psi^m=0$, and therefore
	$K_\varphi^m=0$ by Lemma~\ref{iso}, contradicting the definition of
	$m$. Hence  $\varphi$ is gauge $n$-trivial. For Point~(1), the equality $F_*(K_\varphi)=0$ implies
	\[
	F_*(K_\varphi^n)=0
	\qquad\text{for every }n\geqslant1.
	\]
	Point~(2) therefore gives $K_\varphi^n=0$ for every $n$. 
	Theorem~\ref{Kaledin wg}~(1) concludes the proof.
\end{proof}

\section{\textcolor{bordeau}{Formality transmission}}\label{section 2}

\noindent In this section, we apply the transmission criterion developed in Section \ref{section 1} to establish the formality transmission theorem for
algebras encoded by colored operads or properads. In this particular context, we adapt the construction to deal with any commutative ground ring.

\begin{assumption}\label{assump}
Let us fix $R$ a commutative ring and let $A$ be a chain complex over $R$ related to $H(A)$ via a \emph{contraction}: there exists \[
\hbox{
	\begin{tikzpicture}
		
		\def\upshift{0.075}
		\def\downshift{0.075}
		\pgfmathsetmacro{\midshift}{0.005}
		
		\node[left] (x) at (0, 0) {$(A,d)$};
		\node[right=1.5 cm of x] (y) {$(H(A),0)$};
		
		\draw[->] ($(x.east) + (0.1, \upshift)$) -- node[above]{\mbox{\tiny{$p$}}} ($(y.west) + (-0.1, \upshift)$);
		\draw[->] ($(y.west) + (-0.1, -\downshift)$) -- node[below]{\mbox{\tiny{$i$}}} ($(x.east) + (0.1, -\downshift)$);
		
		\draw[->] ($(x.south west) + (0, 0.1)$) to [out=-160,in=160,looseness=5] node[left]{\mbox{\tiny{$h$}}} ($(x.north west) - (0, 0.1)$);
\end{tikzpicture}} \] such that $ip - \mathrm{id}_A = d_A h + h d_A\ ,$  $pi = \mathrm{id}_{H(A)} \ ,$ $h^2 =0\ , $ $  ph = 0\ ,$ $hi =0 \ .   $ Let $\C$ be a reduced weight-graded dg coproperad over $R$, \[\C =   \mathrm{I} \oplus \C^{(1)} \oplus \C^{(2)} \oplus  \cdots \oplus \C^{(n)} \oplus \cdots \quad \mbox{and} \quad d_{\C} \left(\C^{(k)}\right) \subset \C^{(k - 1)} \ .\]  Let $\overline{\C}$ be the coaugmentation coideal . Let us consider the properad 
\begin{equation}\label{operad P}
\P \coloneqq \mathcal{T}\left(s^{-1}\C^{(1)}\right) / \left(d_{\Cobar \C} \left(s^{-1}\C^{(2)}\right)\right) \ , 
\end{equation} which comes with a twisting morphism $\C \to \P$.   
\end{assumption}

\begin{example}
	Among possible choices for such data, we have: 
	\begin{enumerate}
		\item The bar construction $ \Bar \P$ of a reduced augmented properad in $\mathrm{gr Mod}_R$, see \cite{PHC}.
		\item The Koszul dual coproperad $\C := \P^{\antishriek}$, if $\P$ is a homogeneous Koszul properad, see \cite{Vallette_2007}, or an inhomogeneous quadratic Koszul properad, see \cite[Appendix~A]{GCTV12}.		
		\item Over a field, the quasi-planar cooperad $\C = \Bar \left(\mathcal{E} \otimes \P\right) $ where $\mathcal{E}$ denotes the Barratt-Eccles operad and $\P$ is any reduced augmented operad, see \cite{LRL23} \ . 
	\end{enumerate}
\end{example}

\subsection{Properadic homotopical calculus}\label{2.1}

In this section, we briefly recall the main ingredients of the \emph{properadic calculus} developed in \cite{PHC, SPH, graphexp}, which governs the homotopical properties of algebraic structures encoded by a properad.

\begin{definition}[Convolution Lie-admissible algebra and $\Cobar \C$-algebra structure]
	The \emph{convolution dg Lie-admissible algebra} associated to $(\C, A)$ is the complete dg Lie-admissible algebra \[\mathfrak{g}_{A} := \left( \Hom_{\mathbb{S}} \left(\overline{\C}, \mathrm{End}_A\right), \star, d , \mathcal{F} \right) \ , \] 	whose underlying space is \[\Hom_{\mathbb{S}}\left(\overline{\C}, \End_A \right) \coloneqq \prod_{m,n \geqslant0} \Hom_{\mathbb{S}_m^{op} \times \mathbb{S}_n}\left(\overline{\C}(m,n),\End_A(m,n) \right)\ .\] It is equipped with the Lie-admissible product \[\varphi \star \psi \coloneqq \overline{\C} \xrightarrow{\Delta_{(1,1)}} \overline{\C} \underset{(1,1)}{\boxtimes} \overline{\C} \xrightarrow{\varphi \underset{(1,1)}{\boxtimes}  \psi} \End_A \underset{(1,1)}{\boxtimes} \End_A \xrightarrow{\gamma_{(1,1)}} \End_A \ , \] the differential \[d (\varphi) \coloneqq d_{\End_A} \circ \varphi - (-1)^{|\varphi|} \varphi \circ d_{\overline{\C}} \] and the complete descending filtration given by \[
		\mathcal{F}^n \mathfrak{g}_{A} \coloneqq \prod_{ k \geqslant n  } \Hom_{\mathbb{S}}\left(  \C^{(k)},\End_A \right) \ . \]  There is a natural embedding into the dg Lie-admissible algebra of all maps from $\C$, \[\mathfrak{g}_{A} \hookrightarrow \mathfrak{a}_{A} \coloneqq \left( \Hom_{\mathbb{S}} \left(\C, \mathrm{End}_A\right),  d \right) \ . \]   An \emph{$\Cobar \C$-algebra structure} $(A,\varphi)$  is a Maurer--Cartan element in $\mathfrak{g}_A$.  We refer the reader to \cite{Merkulov_2009} for more details. 	
\end{definition}

\begin{definition}[$\infty$-morphism]
	Let $(A,\varphi)$ and $(B,\psi) $ be two $\Cobar \C$-algebra structures and let \[ f  \in  \mathfrak{a}^A_B \coloneqq \left(\Hom_{\mathbb{S}} \left( \C , \End^A_B\right),  d \right)\ .\] We equip $\mathfrak a_B^A$ with the same weight filtration than $\mathfrak g_B^A$. The \emph{left action} of $\psi \in  \mathfrak{a}_B$ on $f$ and the \emph{right action} of $\varphi \in  \mathfrak{a}_A$ on $f$ are \[\begin{array}{ccccccccc}
		\psi 	\lhd f & \coloneqq & \overline{\C} & \xrightarrow{\Delta_{(1,*)}} & \overline{\C} \underset{(1,*)}{\boxtimes} \C &  \xrightarrow{\psi \underset{(1,*)}{\boxtimes} f} &  \End_B \underset{(1,*)}{\boxtimes} \End_B^A & \longrightarrow & \End_B^A \ ,  \\
		& & \mathcal{I} & \overset {  \cong }{ \ \longrightarrow \ } & \mathcal{I} \ \boxtimes \ \mathcal{I} &  \xrightarrow{\psi \; \boxtimes \; f} & \End_B \ \boxtimes \ \End_B^A & \longrightarrow & \End_B^A \ , \\ 
		& & & & & & & & \\ f 	\rhd \varphi & \coloneqq & \overline{\C} & \xrightarrow{\Delta_{(*,1)}} & \C \underset{(*,1)}{\boxtimes} \overline{\C}  &  \xrightarrow{f \underset{(*,1)}{\boxtimes} \varphi} &  \End_B^A \underset{(*,1)}{\boxtimes} \End_A & \longrightarrow & \End_B^A \ ,  \\
		& & \mathcal{I} & \overset{\cong}{\ \longrightarrow \ } & \mathcal{I} \ \boxtimes \ \mathcal{I} &  \xrightarrow{f \; \boxtimes \; \varphi} & \End_B^A \ \boxtimes \End_A & \longrightarrow & \End_B^A \ .
	\end{array}\]  An \emph{$\infty$-morphism} $\varphi \rightsquigarrow \psi$ is a degree-zero morphism $f \in \mathfrak{a}^A_B$ such that \[f 	\rhd \varphi -  \psi	\lhd f = d(f) \  .\]  Such an  $\infty$-morphism $f : \varphi \rightsquigarrow \psi$ is called
	\begin{itemize}
		\item[$\centerdot$] an \emph{$\infty$-quasi-isomorphism} if $f^{(0)}: A \to B$ is a quasi-isomorphism; 
		\item[$\centerdot$] an \emph{$\infty$-isomorphism} if $f^{(0)} : A \to B$ is an isomorphism; 
		\item[$\centerdot$] an \emph{$\infty$-isotopy} if $A = B$ and $f^{(0)} = \mathrm{id}_A$~. 
	\end{itemize}
\end{definition}

\begin{theorem}[{\cite[Proposition~3.19 and Theorem~3.22]{PHC}}]\label{invertibility prop} \leavevmode
	
	\begin{enumerate}
		\item The composite of degree-zero maps $f \in \mathfrak{a}_B^A$ and $g \in  \mathfrak{a}_C^B$ is defined by  \[g \circledcirc f \coloneqq \C \xrightarrow{\ \Delta \ } \C \ \boxtimes \ \C \xrightarrow{g \ \boxtimes \  f} \End_C^B \ \boxtimes \ \End_B^A \xrightarrow{\ \gamma \ } \End_C^A \ . \]
		\item  The composite of $f : \varphi \rightsquigarrow \psi$ and $g : \psi \rightsquigarrow \phi$ is an $\infty$-morphism \[g \circledcirc f :  \varphi  \rightsquigarrow \phi \ .\]
		\item If $f^{(0)}$ is an isomorphism, $f$ admits a unique inverse with respect to $\circledcirc$~, denoted $f^{-1}$~. 
	\end{enumerate}	
\end{theorem}

\begin{theorem}[{\cite[Theorems~2.22 and 2.29]{graphexp}}]\label{graphexp2}
	Suppose that $R$ is a $\mathbb{Q}$-algebra. The gauge group associated to the convolution dg Lie admissible algebra $\mathfrak{g}_A$ and the group of $\infty$-isotopies are isomorphic through the \emph{graph exponential/logarithm maps}, \[\mathrm{exp} : ((\mathfrak{g}_A)_0, \mathrm{BCH}, 0) \cong (\infty-\mathrm{iso}, \circledcirc, 1) : \mathrm{log}\ . \] In particular, $\Omega \C$-algebra structures are gauge equivalent
	if and only if they are $\infty$-isotopic.
\end{theorem}

\subsection{Gauge formality and properadic Kaledin classes}\label{2.2}

In this section, we recall the notion of gauge formality, the central theme of this paper, and  we recall the properadic Kaledin classes \cite[Section~2.6]{CE24b} that provide a complete characterization of this problem.

\begin{definition}[Gauge homotopy equivalence]
Two $\Omega \C$-algebras $(A,\varphi)$ and $(B,\psi)$ are said 
	\begin{itemize}
		\item[$\centerdot$] \emph{homotopy equivalent} if there exists a zig-zag of quasi-isomorphisms of $\Cobar \C$-algebras \[(A, \varphi) \; \overset{\sim}{\longleftarrow} \;  \cdot \;  \overset{\sim}{\longrightarrow} \;  \cdots \;  \overset{\sim}{\longleftarrow} \; \cdot \;  \overset{\sim}{\longrightarrow} \; (B, \psi) \ .\]
		\item[$\centerdot$] \emph{gauge homotopy equivalent} if there exists an $\infty$-quasi-isomorphism of $\Cobar \C$-algebras 
		\[\begin{tikzcd}[column sep=normal]
			(A,\varphi) \ar[r,squiggly,"\sim"] 
			& (B, \psi) \ .
		\end{tikzcd} 
		\]
		\item[$\centerdot$] \emph{gauge $n$-homotopy equivalent} if $(A,\varphi)$ is gauge homotopy
		equivalent to an $\Omega \C$-algebra $(B,\phi)$ such that $\phi - \psi \in \mathcal{F}^{n+2} \mathfrak{g}_B$.
	\end{itemize}
\end{definition}

\noindent Over a field, homotopy equivalence and gauge homotopy
equivalence coincide in characteristic zero
\cite[Theorem~1.16]{PHC}, for quasi-planar cooperads
\cite[Proposition~43]{LRL23}, and for non-symmetric
cooperads \cite[proof of Proposition~2.21]{CE24b}.
In the non-symmetric case, gauge homotopy equivalence
also implies homotopy equivalence over a commutative ring
whenever $C$ and $\Omega C$ are aritywise and degreewise flat. Any $\Cobar \C$-algebra $(A,\varphi)$ has an induced $\P$-algebra structure 
$(H(A),\varphi_*)$,	see \cite[Proposition~3.11]{Kaledin}.

\begin{definition}[Gauge formality]
	An $\Cobar \C$-algebra structure $(A, \varphi)$ is said to be
	
	\begin{enumerate}
		\item[$\centerdot$] \emph{formal} if it is homotopy equivalent to $(H(A), \varphi_*)$; 
		\item[$\centerdot$] \emph{gauge formal} if it is gauge homotopy equivalent to $(H(A),\varphi_*)$;
		\item[$\centerdot$] \emph{gauge $n$-formal} if it is gauge $n$-homotopy equivalent to $(H(A),\varphi_*)$.
	\end{enumerate} 
\end{definition}

\begin{theorem}[Homotopy transfer theorem] \label{HTT}  
	Suppose that \(\C\) is one of the following:
\begin{enumerate}
	\item a connected non-symmetric cooperad over \(R\);
	\item a connected symmetric cooperad over a \(\mathbb Q\)-algebra \(R\);
	\item a conilpotent quasi-planar cooperad over a field \(R\);
	\item a conilpotent coproperad over a field of characteristic zero.
\end{enumerate}
Any \(\Omega \C\)-algebra \((A,\varphi)\)
admits a transferred structure \((H(A),\varphi_t)\), with
\(\varphi_t^{(1)}=\varphi_*\), and \(\infty\)-quasi-isomorphisms whose weight-zero components are \(i\) and \(p\).
\[
i_\infty:(H(A),\varphi_t)\rightsquigarrow(A,\varphi),
\qquad
p_\infty:(A,\varphi)\rightsquigarrow(H(A),\varphi_t) \ .
\]
They may be chosen so that
\(p_\infty\circledcirc i_\infty=\mathrm{id}_{H(A)}\).
\end{theorem}

\begin{proof} 	For the first two cases, see \cite[Theorem~1.3]{berglund};
	for the third, see \cite[Theorem~8 and Proposition~43]{LRL23};
	and for the fourth, see \cite[Theorem~4.14]{PHC}.
	For any such pair, \(c=p_\infty\circledcirc i_\infty\) has
	\(c^{(0)}=\mathrm{id}_{H(A)}\), hence is an \(\infty\)-isotopy.
	Replacing \(p_\infty\) by \(c^{-1}\circledcirc p_\infty\)
	gives the last assertion. 
\end{proof}

\begin{corollary}[{\cite[Proposition 2.18]{Kaledin}}]\label{gauge formality} An $\Cobar \C$-algebra $(A, \varphi)$  admitting a transferred structure $(H(A), \varphi_t)$ is gauge formal if and only if there exists an $\infty$-isotopy  \begin{center}
		\begin{tikzcd}[column sep=normal]
			f : (H(A), \varphi_t) \ar[r,squiggly,"="] 
			& (H(A), \varphi_*) \ .
		\end{tikzcd} 
	\end{center} It is \emph{gauge $n$-formal} if and only if there exist an $\Cobar \C$-algebra $(H(A),\psi)$  and an $\infty$-isotopy 	\begin{center}
	\begin{tikzcd}[column sep=normal]
		f : (H(A), \varphi_t) \ar[r,squiggly,"="] 
		& (H(A), \psi) \ ,
	\end{tikzcd} 
	\end{center} such that $\psi^{(k)} = 0$, for all $2 \leqslant k \leqslant n + 1$. 
\end{corollary}

\begin{proof}
	Suppose that \(u:(A,\varphi)\rightsquigarrow(H(A),\vartheta)\)
	realizes gauge formality, or gauge \(n\)-formality. The composite
	\(v=u\circledcirc i_\infty\) is an \(\infty\)-isomorphism.
	Its linear term \(a\) is an automorphism of \((H(A),\varphi_*)\).
	Then \(a^{-1}\circledcirc v\) is an \(\infty\)-isotopy with
	target structure \(a^{-1}\cdot\vartheta\). Conjugation by \(a^{-1}\)
	preserves \(\varphi_*\) and the required vanishing of the higher
	weights. Conversely, compose the isotopy in the statement with
	\(p_\infty\). See also \cite[Proposition~2.17]{Kaledin}.
\end{proof}

\noindent Hence, if $R$ is a $\mathbb{Q}$-algebra, $(A, \varphi)$ is gauge formal if and only if $\varphi_t$ is gauge trivial through Theorem \ref{graphexp2}. Hence, we have a formality transmission theorem applying Theorem \ref{thm gauge transmission 2} to the  weight-graded dg Lie algebra $\mathfrak{g}_{H}$~. If $R$ is not a $\mathbb{Q}$-algebra, the gauge group does not even exist. However, we can still ask for the existence of an $\infty$-isotopy and formality transmission is generalized in Section \ref{2.5} to deal with any commutative ground ring. 

\begin{definition}[Properadic Kaledin classes]
	Let $(H,\varphi)$ be an $\Cobar \C$-algebra structure on a graded $R$-module $H$. Let us set $\Phi \coloneqq  \mathfrak{D}_1(\varphi)$. 
	\begin{enumerate}
		\item Its \emph{Kaledin class} $K_{\varphi}$ is the Kaledin class of the prismatic decomposition $\Phi$,  \[K_{\varphi} =  \left[\partial_{\hbar}   \Phi \right] \in H_{-1}\left(\mathfrak{g}_H[\![\hbar]\!]^{\Phi}\right) \ . \]
		\item  Its \emph{$n^{\text{th}}$-truncated Kaledin class} $K_{\varphi}^n$ is the $n^{\text{th}}$-truncated Kaledin class of $\Phi$,  \[ K^n_{\varphi} = \left[\varphi^{(2)} + 2 \varphi^{(3)} \hbar + \dots + n \varphi^{(n+1)} \hbar^{n-1} \right]  \in H_{-1}\left(\left(\mathfrak{g}_H[\![\hbar]\!]/(\hbar^{n })\right)^{\overline{\Phi}^n}\right) \ .  \] 
	\end{enumerate}
\end{definition}

\begin{theorem}[{\cite[Theorem 2.48]{CE24b}}]\label{B prop}
Let $(A, \varphi)$ be an $\Cobar \C$-algebra structure such that there exists a transferred structure $(H(A),\varphi_t)$.
	\begin{enumerate}

		\item If $R$ is a $\mathbb{Q}$-algebra, the algebra $(A, \varphi)$ is gauge formal if and only if  $K_{\varphi_t}= 0$.
		
		\item If $n \geqslant 1$ and $n ! \in R^{\times}$, the algebra $(A, \varphi)$ is gauge $n $-formal if and only if $K^n_{\varphi_t} = 0$.
\end{enumerate}\end{theorem}

\subsection{The adjoint action of $\infty$-isomorphisms}\label{2.3}
In this section, we recall the adjoint action of $\infty$-isomorphisms introduced in \cite{CE24b}.

\begin{notation}\label{nota}
	Let $(A,\varphi)$ and $(B,\psi)$ be $\Omega \C $-algebra structures and fix homogeneous  elements
	\[  f, \ell  : \C \to \mathrm{End}_B^A , \quad  \mbox{and} \quad  g , h : \C \to \mathrm{End}_C^B   \ .   \]  Recall the notation $g \circledcirc (f; \ell)$ of \cite{PHC}  for $g$ applied everywhere at the bottom level and $f$ applied everywhere at the top level except at one place where $\ell$ is applied. The operation $(g ; h)   \circledcirc f$ is defined similarly.  Let us note that we have \[\psi \lhd f = \left(1 ; \psi  \right)\circledcirc f, \quad  f \rhd \varphi = f \circledcirc \left(1 ; \varphi \right), \quad  x \star y = (1; x) \circledcirc \left(1 ; y \right) \ . \]
\end{notation}

\begin{proposition}[{\cite[Propositions 2.15 and 2.16]{CE24b}}]\label{Ad_f}
	If $f  \in \mathfrak{a}_B^A $ is invertible under $\circledcirc$, let us denote $\mathrm{Ad}_f \left( \varphi  \right) \coloneqq (f ; f \rhd  \varphi) \circledcirc f^{-1}$ and 
	\[ f \cdot \varphi \coloneqq \mathrm{Ad}_f \left(  \varphi  \right) -  \left(f ;d(f) \right) \circledcirc f^{-1} \ . \]  
	\begin{enumerate}
		\item The element $f \cdot \varphi$ defines an $\Omega \C $-algebra structure on $B$ and $f$ is an $\infty$-morphism  \[\begin{tikzcd}[column sep=normal]
			f : (A, \varphi)   \ar[r,squiggly,"\sim"] 
			& (B, f \cdot \varphi)  \ .
		\end{tikzcd} 
		\] 
		\item If $g \in \mathfrak{a}_C^B $ is invertible under $\circledcirc$, we have $g \cdot (f \cdot \varphi) = (g \circledcirc f) \cdot \varphi \ .$
		\item 	It induces an isomorphism  of dg Lie algebras \[\mathrm{Ad}_f : \mathfrak{g}_A^{\varphi} \longrightarrow \mathfrak{g}_B^{f  \cdot \varphi} \ ,\]   whose inverse is given by $\mathrm{Ad}_{f^{-1}}$~. 
	\end{enumerate}
\end{proposition}

\begin{proposition}\label{twisted-differential}
	Let $f : (A, \varphi) \rightsquigarrow (B, \psi)$ be an $\infty$-morphism between $\Omega \C$-algebra structures. The operator  $d^f : \mathfrak{a}_B^A \to \mathfrak{a}_B^A$  given by  \[d^f (x) \coloneqq d(x) + \psi \lhd (f; x) - (-1)^{|x|} (f;x) \rhd \varphi   \]  defines a twisted differential on $\mathfrak{a}_B^A$. 
\end{proposition}

\begin{proof}
	For a degree-zero map $u$, set $Q(u):=d(u)+\psi\triangleleft u-u\triangleright\varphi.$ 	The Leibniz rule, the identities in the proof of	\cite[Theorem~5.1]{PHC}, and the Maurer--Cartan equations give
	\[
	dQ(u)+\psi\lhd \left(u;Q(u) \right)
	+\left(u;Q(u)\right)\rhd \varphi=0.
	\]
	The cancellations involved are valid over any commutative ring. Let $x$ be homogeneous and put $q=|x|$. Introduce a graded
	parameter $\varepsilon$ satisfying
	$\varepsilon^2=0$, $|\varepsilon|=-|x|$ and $d\varepsilon=0$.
	For $u=f+\varepsilon x$, the $\infty$-morphism equation
	$Q(f)=0$ gives
	\[
	Q(u)
	=\varepsilon\left(
	(-1)^{|x|} d(x)+(-1)^q\psi\lhd(f;x)
	-(f;x)\rhd \varphi
	\right)
	=(-1)^{|x|}\varepsilon\,d^f(x).
	\]
	Taking the coefficient of $\varepsilon$ in the preceding
	identity therefore yields
	\[
	0=d\left(d^f(x)\right)
	+\psi\lhd\left(f;d^f(x)\right)
	+(-1)^{|x|}\left(f;d^f(x)\right)\rhd \varphi
	=(d^f)^2(x) \ . \qedhere
	\]
\end{proof}

\begin{lemma}[Invariance of Kaledin classes under $\infty$-isomorphisms]\label{piquant prop} Let $f : (A,\varphi) \rightsquigarrow (B, \psi)$ be an $\infty$-morphism between two $\Omega \C$-algebra structures on graded $R$-modules. Let us set \[\Phi \coloneqq \mathfrak{D}_1(\varphi)~, \quad \Psi \coloneqq \mathfrak{D}_1(\psi) \quad \mbox{and} \quad F \coloneqq \mathfrak{D}_0(f) \ .\] 
	\begin{enumerate}
		\item The following identity holds \[F \rhd \partial_{\hbar} \Phi - \partial_{\hbar} \Psi \lhd F  = d^{F}\left(\partial_{\hbar} F \right)  \ . \]
		\item If $f$ is an $\infty$-isomorphism, the classes $K_{\Psi}$ and $\mathrm{Ad}_{F}\left(K_{\Phi}\right)$ are equal in \[H_{-1}\left(\mathfrak{g}_{B}[\![\hbar]\!]^{\Psi}\right) \ .\]  
		\item The same compatibility relation holds modulo $\hbar^n$, for every $n \geqslant 1$ : the classes $K_{\psi}^n$ and $\mathrm{Ad}_{F}^n(K_{\varphi}^n)$ are equal in homology.  
	\end{enumerate}
\end{lemma}

\begin{proof}
	Points~(2) and~(3) follow from \cite[Lemma~2.47]{CE24b}. For point~(1), since $f$ is an $\infty$-morphism, we have
	\[
	F \rhd \Phi - \Psi \lhd F = d(F) \ .
	\]
	Differentiating with respect to $\hbar$, and using the Leibniz rule, yields
	\[
	\left(F;\partial_{\hbar}F\right)\rhd\Phi
	+F\rhd\partial_{\hbar}\Phi
	-\partial_{\hbar}\Psi\lhd F
	-\Psi\lhd\left(F;\partial_{\hbar}F\right)
	=d\left(\partial_{\hbar}F\right) \ .
	\]
	This is precisely the identity in point~(1).
\end{proof}

\subsection{Properadic spectral sequences}\label{2.4}

The purpose of this section is to introduce the properadic spectral sequence associated with an $\infty$-morphism and study its homotopy invariance property.

\begin{definition}[Properadic spectral sequence] Let $f : (A, \varphi) \rightsquigarrow (B, \psi)$ be an $\infty$-morphism. 
	\begin{enumerate}
		
		\item 	The \emph{properadic spectral sequence} of $f$ is the spectral sequence \[\left(E(\C, A, B, f)^{p,q}_r, d_r \right) \] associated with the filtered complex $\left( \mathfrak{g}^A_B, d^f \right) $.   
		
		\item 	The \emph{properadic spectral sequence} of $(A, \varphi)$ is the spectral sequence \[\left(E(\C, A, \varphi)^{p,q}_r, d_r \right) \] associated with $\mathrm{id}_A$, that is, with the filtered complex $\mathfrak{g}_A^{\varphi} \coloneqq \left( \mathfrak{g}_A, d^{\varphi} \coloneqq d + \left[\varphi, -\right] \right) $. 
	\end{enumerate}
\end{definition}

\begin{lemma}\label{compo}
	Let us consider $\infty$-morphisms between $\Omega \C$-algebra structures \begin{center}
		\begin{tikzcd}[column sep=normal]
		 (A, \varphi) \ar[r,squiggly,"f"] 
			& (B, \psi) \ar[r,squiggly,"g"] 
			& (C, \phi) \ .
		\end{tikzcd} 
	\end{center} 
	The composition $\circledcirc$ induces morphisms of filtered chain complexes
		\begin{center}
			\begin{tikzcd}[column sep=normal]
			g^f_* :	\left(\mathfrak{a}^A_B , d^{f} \right) \ar[rr,squiggly,"g\; \circledcirc \; (f; -)"] 
				&& \left(\mathfrak{a}^A_C , d^{\,g \circledcirc f} \right)  
				&&\ar[ll,squiggly,"{(g;-) \; \circledcirc \; f}", swap] \left(\mathfrak{a}^B_C , d^{g} \right) : f^*_g \ .
			\end{tikzcd} 
		\end{center} 

\end{lemma}

\begin{proof} 
Let $x\in\mathfrak a_B^A$ be homogeneous. Let us prove that
\[	g\circledcirc\bigl(f;d^f(x)\bigr)
	=
	d^{\,g\circledcirc f}\bigl(g\circledcirc(f;x)\bigr)\ . \]
For homogeneous maps $u$ and $v$ of the same type as $f$, let us denote
		\[\mathcal{M}(u,v) \coloneqq g \circledcirc (f; u,v) \]  the signed sum of the graphs obtained by applying the
		comonadic decomposition map, with bottom-level vertices labeled
		by $g$, two distinguished upper vertices labeled by $u$ and $v$, and
		all remaining upper vertices labeled by $f$. Interchanging the two distinguished vertices yields
		\begin{equation}\label{eq:symmetry-M}
		\mathcal{M}(u,v)=(-1)^{|u||v|}\mathcal{M}(v,u).
		\end{equation} The iterated comonadic decomposition gives
		\[(g;g\rhd\psi)\circledcirc(f;x)
		=
		g\circledcirc\bigl(f;\psi\lhd(f;x)\bigr)
		+
			\mathcal{M}(\psi\lhd f,x) \ .\]
		Indeed, the first term consists of those graphs in which the
		$\psi$-labeled vertex lies in the distinguished branch containing
		$x$, whereas the second term consists of those in which
		$\psi\lhd f$ and $x$ lie over distinct inputs of $g$. Similarly, we have
\[
	\bigl(g\circledcirc f;
			g\circledcirc(f;x)\bigr)\rhd\varphi =
			g\circledcirc\bigl(f;(f;x)\rhd\varphi\bigr)
			+
				\mathcal{M}(x,f\rhd\varphi) \ . \] Since $d$ is a coderivation, the following identity holds
	\[
		d\bigl(g\circledcirc(f;x)\bigr)
			= 
			(g;d(g))\circledcirc(f;x)
			+
			g\circledcirc(f;d(x)) +
				\mathcal{M}(d(f),x) \ .\] Moreover, by \cite[Lemma~2.13 (2)]{CE24b}, with $x =1$, $y = \phi$ and $\ell = x$, we have
	\[
		\phi\lhd
		\bigl(g\circledcirc f;
		g\circledcirc(f;x)\bigr)
		=
		(g;\phi\lhd g)\circledcirc(f;x). \] Using the $\infty$-morphism identity of $g$, this leads to
		\[ d^{\,g\circledcirc f}\bigl(g\circledcirc(f;x)\bigr)
		-
		g\circledcirc\bigl(f;d^f(x)\bigr) =
			\mathcal{M}(d(f),x)
		+
			\mathcal{M}(\psi\lhd f,x)
		-
		(-1)^{|x|}	\mathcal{M}(x,f\rhd\varphi) \ . \] Using the $\infty$-morphism identity of $f$, this vanishes since 
		$f\rhd\varphi$ has degree $-1$ and by
		\eqref{eq:symmetry-M},
		\[
		\mathcal{M}(f\rhd\varphi,x)
		=
		(-1)^{|x|}	\mathcal{M}(x,f\rhd\varphi) \ .
		\] The proof that $f^*_g$ commutes with the differentials is completely symmetric. 
\end{proof}

\begin{lemma}\label{formuless}
	Let us consider $\infty$-morphisms between $\Omega \C$-algebra structures \begin{center}
		\begin{tikzcd}[column sep=normal]
			(D, \mu)\ar[r,squiggly,"h"] &	 (A, \varphi) \ar[r,squiggly,"f"] 
			& (B, \psi) \ar[r,squiggly,"g"] 
			& (C, \phi) \ .
		\end{tikzcd} 
	\end{center} The following identities hold: 
	\begin{enumerate}[label=$(\alph*)$]
		\item $f \; \rhd = f^{1}_* $ ;  \hfill $(a')$ $ \lhd \; f = f^*_{1} $ ;  \hfill  \textcolor{white}{;}
		\item $\left(g \circledcirc f\right)_*^h = g_*^{f \circledcirc h} \circ f_*^h$ ; \hfill $(b')$ $\left(f \circledcirc h\right)_g^* = h^*_{g \circledcirc f} \circ f^*_g$ ; \hfill \textcolor{white}{;}  
		\item $h^*_{g \circledcirc f} \circ g_*^f = g_*^{f \circledcirc h} \circ h^*_f $ .
		\end{enumerate}
	If $f$ is invertible under $\circledcirc$, the following propositions hold:
	
	\begin{enumerate}
		\item[$(d)$] $\mathrm{Ad}_f = (f^{-1})^*_f \circ f^1_* = f_*^{f^{-1}} \circ \left(f^{-1}\right)^*_1 $ ;
		\item[$(e)$] the maps $	f^h_*$  and  $f^*_g$  are isomorphisms with respective inverses \[\left(f^{-1}\right)^{f\circledcirc h}_* \quad \mbox{and}  \quad \left( f^{-1} \right)^*_{g \circledcirc f} . \]
	\end{enumerate} 
\end{lemma}

\begin{proof}
	Formulas~$(a)$ and~$(a')$ follow directly from the definitions, since
	\[
	f \; \rhd = f \circledcirc (1;-)
	\qquad\text{and}\qquad
	\lhd f = (1;-) \circledcirc f \ .
	\]
The remaining formulas follow from \cite[Lemma~2.13]{CE24b}: formulas~$(b)$ and~$(b')$ are given by parts~$(2)$ and~$(1)$, respectively, with $x=g$, $g=f$, and $f=h$; formula~$(c)$ follows from part~$(3)$ with $k=h$; and the second equality in~$(d)$ follows from part~$(9)$, the first one being the definition. Finally, let us prove part ($e$). Applying twice formula $(b)$, first with $g = f$, $f = f^{-1}$ and $h =f \circledcirc h$ and second with $g = f^{-1}$, yields \[f_*^h \circ \left(f^{-1}\right)^{f\circledcirc h}_* = 1 = \left(f^{-1}\right)^{f\circledcirc h}_* \circ f_*^h \ . \] Applying twice formula $(b')$, first with $g = g \circledcirc f$, $f = f^{-1}$ and $h =f $ and second with $h = f^{-1}$, yields the proof of the second isomorphism. 
\end{proof}

\noindent For the following theorem, we assume that all underlying complexes admit contractions onto their homology. We also assume that either $R$ is a $\mathbb Q$-algebra, $\C$ is a non-symmetric cooperad, or each $\C^{(w)}_j(m,n)$ is projective  over
$R[\mathbb S_m^{\mathrm{op}}\times\mathbb S_n]$.

\begin{theorem}[Homotopy invariance of properadic spectral sequences]
	\label{homotopy invariance}
	Let us consider $\infty$-morphisms between $\Omega\C$-algebra structures
	\begin{center}
		\begin{tikzcd}[column sep=normal]
			(D,\mu)\ar[r,squiggly,"h","\sim"']
			& (A,\varphi)\ar[r,squiggly,"f"]
			& (B,\psi)\ar[r,squiggly,"g","\sim"']
			& (C,\phi),
		\end{tikzcd}
	\end{center}
	where $h$ and $g$ are $\infty$-quasi-isomorphisms.
	Then, for every $r\geq1$, the following diagram commutes
	and its vertical maps are isomorphisms:
	\[
	\begin{tikzcd}[
		row sep=2.8em,
		column sep=3em,
		ampersand replacement=\&]
		E_r(D)
		\arrow[rr,"{\ (g\circledcirc f\circledcirc h)\rhd \ }"]
		\arrow[d,"\cong"']
		\& \& E_r(g\circledcirc f\circledcirc h)
		\&\& E_r(C)
		\arrow[ll,"{\lhd(g\circledcirc f\circledcirc h)}"']
		\arrow[d,"\cong"]
		\\
		E_r(A)
		\arrow[rr,"{f\rhd}"']
		\&\& E_r(f)
		\arrow[u,"\cong","{g_*^{f\circledcirc h}\circ h_f^*}"']
		\&\& E_r(B)
		\arrow[ll,"{\lhd f}"]
	\end{tikzcd}
	\]
	Here $E_r(A)=E(\C,A,\varphi)_r$ and
	$E_r(f)=E(\C,A,B,f)_r$, with analogous notation
	for the other terms, and the arrows denote the induced maps
	on $E_r$. These isomorphisms are compatible with passage
	to the next page.
\end{theorem}

\begin{proof}
	The contractions make $h^{(0)}$ and $g^{(0)}$ chain homotopy
	equivalences. On the associated graded, composition with $h$
	or $g$ acts by pre- or postcomposition with tensor powers of
	these maps. The assumptions on $R$ and $\C$ ensure that the
	resulting maps remain quasi-isomorphisms after applying
	equivariant $\Hom$. Thus the composition maps involving $h$
	or $g$ induce isomorphisms from $E_1$. Fix $r\geqslant 1$. All maps below are understood on $E_r$,
	and all inverses are taken on this page. Define the left
	and right vertical maps by
	\[
	(h_1^*)^{-1}\circ h_*^1
	\qquad\text{and}\qquad
	(g_*^1)^{-1}\circ g_1^*,
	\]
	respectively. These maps and the middle vertical map
	$g_*^{f\circledcirc h}\circ h_f^*$ are isomorphisms
	by the preceding paragraph. Let us prove that the left square commutes.
	By Lemma~\ref{formuless}~$(c)$,
	\[
	h_f^*\circ f_*^1=f_*^h\circ h_1^*.
	\]
	Using $(a)$ and applying $(b)$ twice, we obtain
	\[
	\begin{aligned}
		&g_*^{f\circledcirc h}\circ h_f^*
		\circ(f\rhd)\circ(h_1^*)^{-1}\circ h_*^1  =
		g_*^{f\circledcirc h}\circ h_f^*\circ f_*^1
		\circ(h_1^*)^{-1}\circ h_*^1\\
		&\qquad =
		g_*^{f\circledcirc h}\circ f_*^h\circ h_*^1  =
		(g\circledcirc f\circledcirc h)\rhd.
	\end{aligned}
	\] 	For the right square, Lemma~\ref{formuless}~$(c)$ gives
	\[
	g_*^{f\circledcirc h}\circ h_f^*
	=h_{g\circledcirc f}^*\circ g_*^f,
	\qquad
	g_*^f\circ f_1^*=f_g^*\circ g_*^1.
	\]
	Using $(a')$ and applying $(b')$ twice similarly yields
	\[ g_*^{f\circledcirc h}\circ h_f^*
		\circ(\lhd f)\circ(g_*^1)^{-1}\circ g_1^*  = \lhd(g\circledcirc f\circledcirc h) \ .
	\]
	Hence the diagram commutes. The isomorphisms are compatible
	with passage to the next page, and the same argument applies
	to the full convolution complexes.
\end{proof}

\subsection{Formality transmission}\label{2.5}

In this section, we prove the formality transmission theorem and present  first two applications : 
\begin{enumerate}
	\item[(A)] How do automorphisms of the transferred structure affect formality?
	\item[(B)] Given an $\infty$-morphism between two $\Cobar \mathcal C$-algebra structures \[\begin{tikzcd}[column sep=normal]
		f : (A, \varphi)   \ar[r,squiggly] 
		& (B, \psi)  \ ,
	\end{tikzcd} 
	\]  when does the formality of one of these algebras imply the formality of the other?
\end{enumerate}

\begin{theorem}[Formality transmission]\label{thm formality transmission}
	Let $(A, \varphi)$ be an $\Cobar \C$-algebra admitting a transferred structure. Let $(\mathfrak{h}, d')$ be a weight-graded chain complex and consider a filtered morphism  \[f : \mathfrak{g}_{H(A)}^{\varphi_t} \longrightarrow (\mathfrak{h}, d') \ .\]  For every $k \geqslant 1$, we consider the map induced on the  properadic spectral sequences:\[ \theta_{k} : E(\C, H(A), \varphi_t)_{k}^{k, -1} \longrightarrow E\left( \mathfrak h \right)_{k}^{k, -1} \ .\] Let $f_*$ be the map induced by the decomposition $\mathcal D_0(f)$ in homology. Then the following hold. \begin{enumerate} \item If $R$ is a $\mathbb{Q}$-algebra,  if $f_*(K_{\varphi_t}) = 0$ and if $\theta_k$ is injective for every $k > 1$, then $(A, \varphi)$ is gauge formal.  
		
		\item   Let $n \geqslant 1$ be an integer such that $n !$ is a unit in $R$. If \(f_*(K_{\varphi_t}^n)=0\) and if \(\theta_k\) is
		injective for every \(1 < k\leqslant n + 1 \), then $(A, \varphi)$ is gauge $n$-formal. 
	\end{enumerate}
\end{theorem}

\begin{proof}
	When $R$ is a $\mathbb{Q}$-algebra, this follows from gauge triviality transmission, that is Theorem \ref{thm gauge transmission 2}, by using Corollary \ref{gauge formality} and the isomorphism of Theorem \ref{graphexp2}. For Point (2), when $R$ is not a $\mathbb{Q}$-algebra, we use the same reasoning by using the properadic Kaledin classes. By Theorem \ref{B prop} (2), it is enough to prove that $K_{\varphi_t}^n=0$. Suppose, for a contradiction, that this class is nonzero and let $m$ be the smallest element such that  \[K_{\varphi_t}^m\neq 0 \ .\] 
	If \(m=1\), set \(g \coloneqq \mathrm{id}\) and \(\psi \coloneqq \varphi_t\). If \(m>1\), then
	\(K_{\varphi_t}^{m-1}=0\) and, by Theorem~\ref{B prop}~(2), there exists an $\Omega \C$-algebra structure $(H(A), \psi)$ such that $\psi^{(k)} = 0$ for $2 \leqslant k \leqslant m$ and an $\infty$-isotopy $g : \psi \rightsquigarrow \varphi_t $~.   By Lemma \ref{piquant prop} (3) applied to $\mathcal{D}_0(g)$, we have
	\[
	f_* \circ  \mathrm{Ad}_{\mathcal{D}_0(g)}(K_{\psi}^m)
	=
	f_*(K_{\varphi_t}^m)
	=
	0 \ .
	\]  Since $g$ induces an isomorphism, \[\mathrm{Ad}_g : \mathfrak{g}_{H(A)}^{\psi} \longrightarrow \mathfrak{g}_{H(A)}^{\varphi_t} \ ,\]  we deduce as in Theorem \ref{thm gauge transmission 2} that
	\(m\psi^{(m+1)}\) is a boundary for \(d^{\varphi_*}\), and hence
	$K_{\psi}^m=0.$	Finally, Lemma \ref{piquant prop} (3) applied to $\mathcal{D}_0(g)$ implies that
	\[ 	K_{\varphi_t}^m 	=
	\mathrm{Ad}_{\mathcal{D}_0(g)}(K_{\psi}^m)
	=
	0 \ , 
	\] which leads to a contradiction. 
\end{proof}

\begin{corollary}[Application A]\label{appli A}
	Let $A$ be an $\Omega \C$-algebra admitting a transferred structure $(H(A), \varphi_t)$. Let $S$ be an $\infty$-automorphism of this transferred structure. Let us consider \[\mathrm{Ad}_S - \mathrm{id} : \mathfrak{g}_{H(A)}^{\varphi_t} \longrightarrow \mathfrak{g}_{H(A)}^{\varphi_t} \] and  for every $k \geqslant 1$, its induced map on the $E_k$ page of the properadic spectral sequences \[ \theta_{k} :  E(\C, H(A), \varphi_t)^{k,-1}_k  \longrightarrow E(\C,  H(A), \varphi_t)^{k,-1}_k      \ . \] 
	\begin{enumerate}
		\item If $R$ is a $\mathbb{Q}$-algebra, and   $\theta_k$ is injective for every $k > 1$, then $(A, \varphi)$ is gauge formal.
		\item Let $n \geqslant 1$ be an integer such that $n !$ is a unit in $R$.  If \(\theta_k\) is
		injective for \(1  <  k\leqslant n + 1 \), then $(A, \varphi)$ is gauge $n$-formal. 
	\end{enumerate}
\end{corollary}

\begin{proof}
	By Lemma \ref{piquant prop} (3),  we have $\left(\mathrm{Ad}_S - \mathrm{id} \right)_*(K_{\varphi_t}) = 0   .$ The result then follows from Theorem \ref{thm formality transmission}. 
\end{proof}

\begin{corollary}[Application B]\label{appli B}
	Let us fix an $\infty$-morphism between $\Cobar \mathcal C$-algebra structures \[\begin{tikzcd}[column sep=normal]
		f : (A, \varphi)   \ar[r,squiggly] 
		& (B, \psi)  \ ,
	\end{tikzcd} 
	\]  admitting transferred structures. For every $k \geqslant 1$, we consider the maps induced by $f_t$ on the $E_k$ page of the properadic spectral sequences \[ \theta_{k} :  E(\C, H(A), \varphi_t)^{k,-1}_k  \xrightarrow{\ f_t \, \rhd \ } E(\C, H(A), H(B), f_t)^{k,-1}_k   \xleftarrow{\ \lhd \,  f_t \ } E(\C, H(B), \psi_t)^{k,-1}_k : \tilde{\theta}_k  \ . \] 
	\begin{enumerate}
	\item If $R$ is a $\mathbb{Q}$-algebra and $\theta_k$ is injective
		for every $k>1$, then $(A,\varphi)$ is gauge formal
		whenever $(B,\psi)$ is.	
	
	\item Let $n \geqslant 1$ with $n!$ invertible in $R$.
		If $\theta_k$ is injective for $1<k\leqslant n+1$,
		then $(A,\varphi)$ is gauge $n$-formal
		whenever $(B,\psi)$ is.	
\end{enumerate}
	Both statements remain valid with $(A,\varphi)$ and $(B,\psi)$
	interchanged and $\theta_k$ replaced by $\tilde{\theta}_k$.
\end{corollary}

\begin{proof}
	Let $(H(A), \varphi_t)$ and $(H(B), \psi_t)$ be the corresponding transferred structures and define 
	\begin{center}
		\begin{tikzcd}[column sep=normal]
			f_t : (H(A),\varphi_t) \ar[r,squiggly,"i_{\infty}^A "]  & 	(A, \varphi) \ar[r,squiggly,"f"] 
			& (B, \psi)   \ar[r,squiggly," p_{\infty}^B"] 		 
			& (H(B), \psi_t) \ .
		\end{tikzcd} 
	\end{center}  Let us set $\Phi \coloneqq \mathfrak{D}_1(\varphi_t)$~, $\Psi \coloneqq \mathfrak{D}_1(\psi_t)$  and $F \coloneqq \mathfrak{D}_0(f_t) \ .$ By Lemma \ref{piquant prop}, we have \[F \rhd \partial_{\hbar} \Phi - \partial_{\hbar} \Psi \lhd F  = d^{F}\left(\partial_{\hbar} F \right) \ . \]
	
	\begin{enumerate}
		\item If $R$ is a $\mathbb{Q}$-algebra and $(B, \psi)$ is gauge formal, then Theorem \ref{B prop} (1) implies that $K_{\psi_t} = 0$ and we can write $\partial_{\hbar} \Psi = d^\Psi (\Lambda) \ .$ Lemma \ref{compo} implies by $\hbar$-linearity that \[\partial_{\hbar} \Psi \lhd F = d^\Psi (\Lambda) \lhd F = d^F(\Lambda \lhd F)  \  . \]  Thus, $f_*(K_{\varphi_t}) = 0$ and the proof of Point (1) follows from Theorem \ref{thm formality transmission} (1).
		\medskip 
		\item    If $n \geqslant 1$ is such that $n !$ is a unit in $R$ and $(B, \psi)$ is gauge $n$-formal, then  Theorem \ref{B prop} (2) implies that $K_{\psi_t}^n = 0$ and $\partial_{\hbar} \Psi$ is a boundary modulo $(\hbar^n)$. Using Lemma \ref{compo} once again, we get that $F \rhd \partial_{\hbar} \Phi$ is a boundary modulo $(\hbar^n)$ and $f_*(K_{\varphi_t}^n) = 0$. The proof of Point (2) now follows from Theorem \ref{thm formality transmission} (2). 
	\end{enumerate} The proof of the symmetric statements involving  $\tilde{\theta}_k$ is completely symmetric. 
\end{proof}

\begin{remark}\label{error}
	For $ \mathcal{L}_{\infty}$, the above result gives the corrected form of \cite[Theorem~3.4]{Man15}; see also the corrigendum \cite[Section~5]{MR26}. 
	More generally, when $\Omega \C = \P_{\infty}$ is the Koszul resolution of a binary Koszul operad, it yields a corrected version of \cite[Theorem~5.2]{MR19}. 
	The original proofs rely on \cite[Lemma~5.3]{Man15}, whose proof contains an error.
\end{remark}

\subsection{Interlude: Properadic Euler classes}
\label{Euler classes}

We give a properadic version of the Euler-class criterion of \cite{MR19}, which characterizes formality in terms of the associated properadic spectral sequence. We require throughout this subsection that
\[ \C ^{(p)}_j=0 \qquad\text{for }j\ne p. \]
Thus, the weight grading of $\C$ agrees with its homological grading.
This holds, for example, for the bar construction of a reduced augmented properad
concentrated in degree zero, with weight given by the number of vertices. Let $(A,\varphi)$ admit a transferred structure $(H,\varphi_t)$. For an $\Omega C$-algebra structure $(H; \mu)$, let us consider the spectral sequence of $(\mathfrak a_H,d^\mu)$
\[
\left(\widetilde E(C,H,\mu)_r^{p,q},d_r\right) \ .
\]

\begin{definition}[Properadic Euler class]\label{def:properadic-euler}
	The Euler endomorphism of $H$ is
	\[
	e_H(a)=-|a|a
	\]
	for homogeneous $a\in H$, viewed as an element of
	$(\mathfrak a_H^{(0)})_0$.
	Its class in $\widetilde E(C,H,\mu)_2^{0,0}$ is called the Euler class
	of $(H,\mu)$ and is also denoted by $e_H$.
\end{definition}

\begin{remark}\label{rem:properadic-euler-identity}
	For every homogeneous $\psi\in\mathfrak a_H^{(p)}$ of degree $q$,
	\[
	[\psi,e_H]=(p+q)\psi.
	\]
	Indeed, for $c\in C^{(p)}(m,n)$, the map $\psi(c)$ has degree $p+q$;
	the bracket with $e_H$ multiplies it by the difference between total
	output degree and total input degree. This computation applies to any
	number of inputs and outputs.
	Moreover, $d(e_H)=0$, since $H$ has zero differential and the
	coaugmentation coideal of $C$ is preserved by $d_C$. Consequently,
	\begin{equation}\label{eq:properadic-euler-differential}
		d^\mu(e_H)=\sum_{p\ge2}(p-1)\mu^{(p)}\in F^2\mathfrak a_H,
	\end{equation}
	which makes the Euler class well defined.
	For $\Phi=D_1(\mu)$, the corresponding prismatic identity is
	$d^\Phi(e_H)=\hbar\partial_\hbar\Phi$.
\end{remark}

\begin{proposition}\label{euler classes}
	Let $n\geqslant 1$ be such that $n!$ is a unit in $R$.
	For an $\Omega C$-algebra structure $(H,\mu)$ on a graded $R$-module,
	the following are equivalent:
	\begin{enumerate}
		\item $(H,\mu)$ is gauge $n$-formal;
		\item $K_\mu^n=0$;
		\item the Euler class survives to
		$\widetilde E(C,H,\mu)_{n+2}$.
	\end{enumerate}
	In the last condition, survival means that $d_r(e_H)=0$ successively
	for $r=2,\ldots,n+1$.
\end{proposition}

\begin{proof}
	The equivalence of (1) and (2) is Theorem~\ref{B prop}.
	If (1) holds, Corollary~\ref{gauge formality} gives an $\infty$-isotopy
	$f:(H,\mu)\longrightarrow(H,\nu)$ with
	$\nu^{(p)}=0$ for $2\leqslant p\leqslant n+1$.
	The adjoint action of $f$ gives a filtered isomorphism
	\[
	\operatorname{Ad}_f:
	(\mathfrak a_H,d^\mu)\xrightarrow{\ \cong\ }
	(\mathfrak a_H,d^\nu)
	\]
	by Lemmas~\ref{compo} and~\ref{formuless}, which apply to the full convolution complexes.
	Since $f^{(0)}=\operatorname{id}_H$,
	it induces the identity on the associated graded and therefore preserves
	the Euler class. Equation~\eqref{eq:properadic-euler-differential}
	gives $d^\nu(e_H)\in F^{n+2}\mathfrak a_H$, proving (3). Conversely, (3) means that there exists a degree-zero element
	$r=r^{(1)}+\cdots+r^{(n)}$ such that
	\[
	d^\mu(e_H-r)\in F^{n+2}\mathfrak a_H.
	\]
	Taking weight $p$, for $2\leqslant p\leqslant n+1$, yields
	\[
	d\bigl(r^{(p-1)}\bigr)
	+\sum_{j=1}^{p-1}[\mu^{(j)},r^{(p-j)}]
	=(p-1)\mu^{(p)}.
	\]
	This proves that $K_\mu^n=0$ since these identities give
	\[
	d^\Phi\bigl(D_1(r)\bigr)
	\equiv\partial_\hbar D_1(\mu) \pmod{\hbar^n} \ . \qedhere
	\] 
\end{proof}

\begin{theorem}\label{degeneration}
	Let $(A,\varphi)$ admit a transferred structure $(H(A),\varphi_t)$,
	and assume that $C^{(p)}$ is concentrated in homological degree $p$.
	\begin{enumerate}
		\item If $R$ is a $\mathbb Q$-algebra, then $(A,\varphi)$ is gauge
		formal if and only if
		$\widetilde E(C,H(A),\varphi_t)$ degenerates at $E_2$.
		\item Let $n\geqslant 1$ with $n!$ invertible in $R$.
		Then $(A,\varphi)$ is gauge $n$-formal if and only if this
		spectral sequence satisfies $d_r=0$ for $2\leqslant r\leqslant n+1$.
	\end{enumerate}
\end{theorem}

\begin{proof} \leavevmode
	\begin{enumerate}
		\item 	Suppose first that $(A,\varphi)$ is gauge $n$-formal.
		By Corollary~\ref{gauge formality} and the filtered adjoint isomorphism used above, we may
		replace $\varphi_t$ by a structure $\nu$ with
		$\nu^{(p)}=0$ for $2\leqslant p\leqslant n+1$.
		The differential $d^\nu$ then has only a weight-one component and
		components raising weight by at least $n+2$.
		Hence $d_r=0$ for $2\leqslant r\leqslant n+1$.
		Conversely, this vanishing implies that the Euler class survives to
		$\widetilde E_{n+2}$, so Proposition~\ref{euler classes}
		and Theorem~\ref{B prop} imply gauge $n$-formality.
		
		\item  If $R$ is a $\mathbb Q$-algebra and $(A,\varphi)$ is gauge formal,
		the same argument replaces $\varphi_t$ by $\varphi_*$, whose twisted
		differential has weight one; the spectral sequence therefore degenerates
		at $E_2$. Conversely, degeneration implies $K_{\varphi_t}^n=0$ for every
		$n\geqslant 1$. Theorems~\ref{Kaledin wg}~(3) and~\ref{B prop} give $K_{\varphi_t}=0$ and gauge
		formality. \qedhere
	\end{enumerate}

\end{proof}

\subsection{Formality transmission in the colored setting}\label{sec:colored-setting}

The constructions of the preceding sections admit a natural extension
to colored operads. The underlying principle is that the arguments
are formulated in terms of dg convolution algebras and their filtrations, whose colored analogues satisfy the same identities. We briefly describe the necessary changes of framework, without repeating the proofs.

\medskip

\noindent Let $\mathbb{V}$ be a groupoid. The usual set-colored
setting is recovered by taking $\mathbb{V}$ to be discrete.
We refer to \cite{War19} for the corresponding Koszul duality
theory; see also \cite[Section~5]{RL2}. Let $\C$ be a reduced, coaugmented,
conilpotent $\mathbb{V}$-colored dg cooperad over $R$, and write
$\overline{\C}$ for its coaugmentation coideal.

\begin{example}[Endomorphism operad]
	A $\mathbb{V}$-module is a functor $A \colon \mathbb{V} \longrightarrow \mathrm{dgMod}_R$.
	Its colored endomorphism operad is given by
	\[
	\End_A(v_0;v_1,\ldots,v_r)
	:=
	\Hom_R\bigl(
	A(v_1)\otimes\cdots\otimes A(v_r),
	A(v_0)
	\bigr).
	\]
	Here $v_0$ denotes the output color and $v_1,\ldots,v_r$
	the input colors. Composition is defined by substitution along
	matching colors. The symmetric structure is induced by permutations
	of the inputs, with the usual Koszul signs, while the groupoid
	structure is induced by the functoriality of $A$.
\end{example}

\begin{definition}[Colored convolution algebra]
	Let $\mathcal{S}_{\mathbb{V}}$ denote the groupoid of color
	profiles, whose objects are tuples $(v_0;v_1,\ldots,v_r)$
	and whose morphisms combine permutations of the inputs with
	isomorphisms of colors. The colored convolution complex is
	\[
	\mathfrak{g}_A
	:=
	\Hom_{\mathcal{S}_{\mathbb{V}}}
	\bigl(\overline{\C},\End_A\bigr).
	\]
	Here $\Hom_{\mathcal{S}_{\mathbb{V}}}$ denotes the internal
	Hom of colored collections. Thus an element is a family of
	graded maps
	\[
	\varphi_{v_0;v_1,\ldots,v_r}
	\colon
	\overline{\C}(v_0;v_1,\ldots,v_r)
	\longrightarrow
	\End_A(v_0;v_1,\ldots,v_r)
	\]
	compatible with permutations and all isomorphisms in
	$\mathbb{V}$. This complex carries the convolution pre-Lie product
	\[
	\varphi\star\psi
	:=
	\gamma_{(1)}
	\circ
	(\varphi\circ_{(1)}\psi)
	\circ
	\overline{\Delta}_{(1)},
	\]
	where $\overline{\Delta}_{(1)}$ is the reduced infinitesimal
	decomposition of $\C$ and $\gamma_{(1)}$ is the infinitesimal
	composition of $\End_A$, both taken in the colored setting.
	Its differential is
	\[
	d(\varphi)
	=
	d_{\End_A}\circ\varphi
	-
	(-1)^{|\varphi|}
	\varphi\circ d_{\overline{\C}} \ .
	\]
\end{definition}

\noindent We equip $\mathfrak{g}_A$ with the filtration induced by $\C$. The only changes are that decompositions
and compositions must respect the colors, and that all maps
must satisfy the relevant equivariance conditions.  All constructions formulated in
terms of the resulting complete filtered convolution algebra
then carry over to this setting. In particular, the
adjoint action of $\infty$-isomorphisms, properadic spectral sequences, formality transmission and Euler classes have the same form.

\begin{remark}[Transferred structures]
	Statements involving transferred structures additionally require
	homotopy retracts in the category of $\mathbb{V}$-modules. Over a field of characteristic zero,
	when every automorphism group $\mathrm{Aut}_{\mathbb{V}}(v)$
	is finite, such retracts onto homology exist. Homotopy transfer
	in this framework is developed in \cite[Section~2.8]{War19}. In the set-colored non-symmetric setting, the transfer formulas
	involve planar trees and require no averaging over symmetry
	groups. They therefore apply over an arbitrary commutative
	ground ring, see \cite[Theorem~5]{DV15}.
\end{remark}

\section{\textcolor{bordeau}{Formality criteria}}\label{section 4}
In this section, we use  formality transmission to revisit and extend several formality criteria, notably beyond characteristic zero and to dg algebras over operads or properads. The discussion is organized into the following subsections, each devoted to a different criterion:

\smallskip

\sommairesection

\medskip

\noindent Throughout this section, we work under Assumptions \ref{assump} and fix
an $\Omega \C$-algebra structure $(A, \varphi)$ admitting a transferred structure
$(H(A), \varphi_t)$.

\subsection{Formality under domination}
By \cite{MSZ23}, rational formality is preserved under domination:
If a closed orientable manifold (resp. rational Poincaré duality space) $X$ receives a map $Y \to X$ from a formal manifold (resp. space) $Y$ that hits a fundamental class, then $X$ is formal. This section aims to extend this 
result to more general coefficient rings and algebraic structures. In this section, we assume that the reduced weight-graded dg coproperad \(\C\) is in fact a cooperad. Let us recall the notion of module over an $\Omega\C$-algebra, following \cite{BM07}.

\begin{definition}[Modules over an $\Omega\C$-algebra]
	Let $\mathrm{Mod}_{\Omega\C}$ be the two-coloured dg operad with
	colours $\mathsf a$ and $\mathsf m$ given by
	\[
	\mathrm{Mod}_{\Omega\C}(c_1,\ldots,c_k;c)=
	\begin{cases}
		\Omega\C(k),
		& c=\mathsf a\text{ and }c_j=\mathsf a\text{ for all }j,\\[2pt]
		\Omega\C(k),
		& c=\mathsf m\text{ and exactly one }c_j=\mathsf m,\\[2pt]
		0,&\text{otherwise}.
	\end{cases}
	\]
	Its differential, symmetric group actions, units and colour-compatible
	compositions are induced by those of $\Omega\C$; see
	\cite[Section~1.5.1]{BM07}. 	For an $\Omega\C$-algebra $(A,\varphi)$, an \emph{$A$-module} is
	a chain complex $M$ together with a
	$\mathrm{Mod}_{\Omega\C}$-algebra structure on $(A,M)$ whose
	restriction to colour $\mathsf a$ is $\varphi$. Equivalently, it is
	specified by morphisms of complexes
	\[
	\mu^M_{k,i}\colon
	\Omega\C(k)\otimes A^{\otimes(i-1)}\otimes M
	\otimes A^{\otimes(k-i)}\longrightarrow M,
	\qquad k\geqslant1,\quad1\leqslant i\leqslant k,
	\]
	satisfying the operadic axioms together with the operations of $A$.
	Distinguishing one input in each operation of $A$ gives its
	\emph{adjoint module}, denoted $A^{\mathrm{ad}}$. Applying the same construction to $\C$ gives a two-coloured
	cooperad $\C_{\mathrm{mod}}$, with a canonical identification
	\[
	\Omega(\C_{\mathrm{mod}})\cong\mathrm{Mod}_{\Omega\C}.
	\]
	We use this presentation to define $\infty$-morphisms.
	An \emph{$\infty$-morphism of $A$-modules}
	$u:M\rightsquigarrow N$ is an $\infty$-morphism of the corresponding
	two-coloured algebras whose algebra-coloured component is the strict
	identity of $A$. Its module-coloured components are degree-zero maps
	\[
	u_{k,i}\colon
	\C(k)\otimes A^{\otimes(i-1)}\otimes M
	\otimes A^{\otimes(k-i)}\longrightarrow N,
	\]
	compatible with equivariance and satisfying
	\[
	u\star\varphi+u\star\mu^M
	-\mu^N\circledcirc(\id_A;u)=du.
	\]
	Here the first term inserts algebra operations in the algebra inputs,
	and the second inserts module operations in the distinguished module
	input. All expressions use the usual Koszul sign convention.
\end{definition}

\begin{remark}
	For $\Omega\C=A_\infty$, these modules are $A_\infty$-bimodules.
	In particular, the retractions used below must respect both module
	actions, up to the specified higher homotopies.
\end{remark}

\begin{definition}[Restriction of scalars]
	Let $f:(A,\varphi)\rightsquigarrow(B,\psi)$ be an $\infty$-morphism
	and let $M$ be a $B$-module. Its \emph{restriction along $f$}, denoted
	$f^*M$, has underlying complex $M$ and module structure
	\[
	\mu^{f^*M}=\mu^M\circledcirc(f;\id_M).
	\]
	Equivalently, this is the unique module structure for which
	$(f,\id_M):(A,f^*M)\rightsquigarrow(B,M)$ is an $\infty$-morphism,
	with module-coloured component concentrated in weight zero.
	Distinguishing an input in the components of $f$ also gives an
	$\infty$-morphism of $A$-modules
	\[
	f^{\mathrm{ad}}:A^{\mathrm{ad}}
	\rightsquigarrow f^*(B^{\mathrm{ad}}) \ .
	\]
\end{definition}

\begin{theorem}[Formality under domination]\label{formality under domination}
	Let \(f:(A,\varphi)\rightsquigarrow(B,\psi)\) be an \(\infty\)-morphism between $\Omega \C$-algebras admitting transferred structures. Suppose that $f$ admits a homotopy retract as an $A$-module, i.e.  there exists an \(\infty\)-morphism of
	\(A\)-modules
	\[
	r:f^*(B^{\mathrm{ad}})
	\rightsquigarrow A^{\mathrm{ad}}
\quad  \mbox{such that} \quad H\bigl(r^{(0)}\circ f^{(0)}\bigr)
	=
	\id_{H(A)} \]  where $r^{(0)}\colon B\longrightarrow A$ denotes the restriction to  \(\C^{(0)}=\mathbb I\). Then the following holds.
	
	\begin{enumerate}
		\item 	If $R$ is a $\mathbb{Q}$-algebra and if $(B, \psi)$ is gauge formal, so is $(A, \varphi)$.
		\item    Let $n \geqslant 1$ be an integer such that $n !$ is a unit in $R$. If $(B, \psi)$ is gauge $n$-formal, so is $(A, \varphi)$. 
	\end{enumerate}
\end{theorem}

\begin{proof}
	Insertion into the distinguished module input defines
	\[
	r_\sharp\colon\mathfrak g_{B}^A\longrightarrow\mathfrak g_A,
	\qquad
	r_\sharp(x):=r\circledcirc(\id_A;x) \ .
	\]
	Since \((\id_A,r)\) is an \(\infty\)-morphism of
	\(\mathrm{Mod}_{\Omega\C}\)-algebras, the same computation as in
	Proposition~\ref{twisted-differential}, applied in the two-coloured
	setting, gives \[r_\sharp\circ d^f=d^\varphi\circ r_\sharp \ .\] Hence \(r_\sharp\) is a filtered morphism of complexes. Let $\rho_k:=E_k(r_\sharp).$ Under the identification of the \(E_2\)-pages associated with the transferred structures, only the weight-zero components of \(f\) and \(r\) contribute. Therefore,
	\[
	\rho_2\circ\theta_2
	=
	H\left(r^{(0)}\circ f^{(0)}\right)\circ(-)
	=
	\id.
	\]
	Thus \(\theta_2\) is split injective. By functoriality of the
	spectral sequences, the splitting propagates to every subsequent
	page: $\rho_k\circ\theta_k=\id$ for every $k\geqslant2.$ In particular, $\theta_k$	is injective for every \(k>1\). The result now follows from
	Corollary~\ref{appli B}.
\end{proof}

\noindent For $A_\infty$-algebras over a field of characteristic zero,
this recovers the bimodule-retraction criterion of
\cite[Theorem~B]{MSZ23}.
The following corollaries give geometric and invariant-theoretic
applications over more general coefficient rings.    For the next three corollaries, fix $n\geqslant1$ with
$n!\in R^\times$. If $R$ is a $\mathbb Q$-algebra, the same
statements hold for gauge formality.

\begin{corollary}[Poincar\'e duality]
	\label{cor:pd-domination}
	Let $f\colon Y\to X$ be a map between spaces homotopy equivalent
	to finite CW complexes, with finitely generated projective
	$R$-cohomology. Suppose that $X$ is an $R$-oriented
	Poincar\'e duality space of dimension $d$ and that
	\[
	f_*c=[X]
	\qquad\text{for some }c\in H_d(Y;R).
	\]
	If $C^*(Y;R)$ is gauge $n$-formal, so is $C^*(X;R)$.
\end{corollary}

\begin{proof}
	The finiteness and projectivity assumptions give contractions
	onto cohomology, hence transferred structures.
	Choosing cycles representing $c$ and $[X]$, we consider
	\[
C^*(Y;R) \xrightarrow{\ \cap c\ }C_{d-*}(Y;R)
	\xrightarrow{\ f_*\ }C_{d-*}(X;R)
	\xleftarrow[\simeq]{\ \cap[X]\ }C^*(X;R).
	\]
	After restriction along $f^*$, these maps lift to morphisms
	of $E_\infty$-modules
	\cite[Theorem~5.13 and Corollaries~5.22 and~5.26]{GTZ12}.
	Inverting the last map, restricting to $A_\infty$-bimodules
	and passing to transferred models gives the module morphism
	required in Theorem \ref{formality under domination}. Indeed, the projection formula gives
	\[
	\operatorname{PD}_X^{-1}f_*(f^*a\cap c)
	=\operatorname{PD}_X^{-1}(a\cap[X])
	=a
	\]
	for every $a\in H^*(X;R)$.
\end{proof}

\begin{remark}
For a map $f:Y\to X$ of degree $d$ between connected closed
$R$-oriented manifolds of the same dimension, with projective
cohomology, take $c=d^{-1}[Y]$ whenever $d\in R^\times$. In particular, if $C^*(Y;\mathbb F_\ell)$ is gauge
$n$-formal and $C^*(X;\mathbb F_\ell)$ is not, with
$1\leqslant n <\ell$, the degree of every map $Y\to X$ is divisible by $\ell$.
\end{remark}

\begin{corollary}[Finite coverings]
	\label{cor:finite-coverings}
	Let $\pi\colon Y\to X$ be a finite covering whose degree
	is a unit in $R$. Suppose that $C^*(Y;R)$ admits a contraction
	onto cohomology. If $C^*(Y;R)$ is gauge $n$-formal,
	so is $C^*(X;R)$.
\end{corollary}

\begin{proof}
	Let $d$ be the degree of $\pi$. The cochain transfer
	is defined by
	\[
	(\operatorname{tr}_\pi b)(\sigma)
	=\sum_{\widetilde{\sigma}}b(\widetilde{\sigma}),
	\]
	where the sum runs over the lifts of $\sigma$.
	Since a lift of a face extends uniquely to a lift of the
	simplex, the Alexander--Whitney formula gives
	\[
	\operatorname{tr}_\pi(\pi^*a\smile b\smile\pi^*c)
	=a\smile\operatorname{tr}_\pi(b)\smile c,
	\qquad
	\operatorname{tr}_\pi\pi^*=d\,\mathrm{id}.
	\]
	Thus $d^{-1}\operatorname{tr}_\pi$ is a strict bimodule
	retraction of $\pi^*$. The contraction onto cohomology
	descends to this direct summand, so Theorem \ref{formality under domination} applies.
\end{proof}

\begin{corollary}[Finite-group invariants]
	\label{cor:finite-invariants}
	Let a finite group $G$ act by strict automorphisms on an
	$\Omega C$-algebra $(A,\varphi)$, and suppose that $|G|$
	is a unit in $R$. Assume that $A$ and $A^G$ admit
	transferred structures. If $A$ is gauge $n$-formal,
	so is $A^G$.
\end{corollary}

\begin{proof}
	The averaging map is a retraction of the inclusion:
	\[
	r=\frac{1}{|G|}\sum_{g\in G}g\colon A\longrightarrow A^G
	\] 	 Since each $g$ fixes
	$A^G$, averaging commutes with every operation whose inputs,
	except one, lie in $A^G$. Thus $r$ is a strict
	$A^G$-module morphism, and Theorem~\ref{formality under domination} applies.
\end{proof}

\noindent Over a field of characteristic $\ell>0$, this applies
whenever $n<\ell$ and $\ell\nmid|G|$. No $G$-equivariant
formality isomorphism for $A$ is assumed.

\begin{corollary}[Quotients of arrangement complements]
	\label{cor:arrangement-quotients}
	\label{ex:arrangement-complements}
	Let $K/\mathbb Q_p$ be a finite extension with residue field
	$\mathbb F_q$. Fix an embedding $K\hookrightarrow\mathbb C$.
	Let $X$ be the complement of a finite affine hyperplane arrangement defined over $K$.
	Let $\ell\neq p$ be a prime and $h$ be the order of $q$ in  $\mathbb F_\ell^\times$. 
	\begin{enumerate}
		\item If $\pi:X\to Z$ is a finite covering of a Hausdorff space of
		degree prime to $\ell$, then $C^*(Z;R)$ is gauge $n$-formal for
		$1\leqslant n<h$ and
		\[
		R=\mathbb Z_\ell
		\quad\text{or}\quad
		R=\mathbb Z/\ell^s\mathbb Z,\qquad s\geqslant1.
		\]
		\item  	If the hyperplanes are individually defined over $\mathbb Q$,
		then $C^*(Z;R)$ is gauge $(\ell-2)$-formal for every odd
		prime $\ell$ not dividing the covering degree. 
	\end{enumerate}	These conclusions hold for $Z=X/G$ whenever
	a finite group $G$ of order prime to $\ell$ acts freely
	on $X$ by homeomorphisms.
\end{corollary}

\begin{proof}
	By \cite{DCH21}, the algebra $C^*(X;\mathbb Z_\ell)$
	is gauge $n$-formal in our convention.
	Reducing the contraction and the resulting $\infty$-isomorphism
	modulo $\ell^s$ gives the same conclusion over
	$\mathbb Z/\ell^s\mathbb Z$. The covering degree is a unit in $R$, as is $n!$ because
	$n<h\leqslant\ell-1$. Corollary~\ref{cor:finite-coverings}
	therefore applies.
	For hyperplanes defined over $\mathbb Q$, choose an auxiliary
	prime $p$ generating $\mathbb F_\ell^\times$, as in
	\cite[Remark~4.7]{DCH21}, then $h=\ell-1$.
\end{proof}

\begin{example}[A non-orientable quotient of a configuration space]
	For $m\geqslant1$, consider
	\[
	X_m=\operatorname{Conf}_m(\mathbb C^\times)
	=\{(z_1,\ldots,z_m):z_i\neq0,\ z_i\neq z_j\text{ for }i\neq j\}
	\]
	and the involution
	\[
	\tau(z_1,\ldots,z_m)
	=\left(-\frac{1}{\overline z_1},\ldots,
	-\frac{1}{\overline z_m}\right).
	\]
	For every odd prime $\ell$, the cochains
	$C^*(X_m/\langle\tau\rangle;R)$ are gauge $(\ell-2)$-formal for
	\[
	R=\mathbb Z_\ell
	\quad\text{or}\quad
	R=\mathbb Z/\ell^s\mathbb Z,\qquad s\geqslant1.
	\]
	Indeed, $X_m$ is the complement of the hyperplanes $z_i=0$ and
	$z_i=z_j$, all defined over $\mathbb Q$. The involution preserves
	$X_m$ and is free, since a fixed point would satisfy $|z_i|^2=-1$.
	The quotient map is a two-sheeted covering, so
	Corollary~\ref{cor:arrangement-quotients} applies.
	Finally, $\tau$ is antiholomorphic and acts on the orientation
	of the connected complex $m$-manifold $X_m$ by $(-1)^m$.
	For odd $m$, the quotient is therefore a non-orientable smooth manifold.
\end{example}

\subsection{Formality under puncturing}

Stasheff's theorem \cite{Sta83} states that a closed simply connected
manifold is rationally formal if and only if its puncture is.
We extend this result to other coefficient rings, Poincar\'e duality
spaces, and \'etale cochains. Throughout this subsection, we use cohomological grading, strictly
unital $A_\infty$-algebras and morphisms, and normalized Hochschild
complexes, with cochain weight equal to arity minus one.

\begin{theorem}[Formality under puncturing]
	\label{thm:pd-transmission}
	Let $H$ be a finite projective connected graded-commutative
	Poincar\'e duality $R$-algebra of dimension $d\geqslant2$,
	with $H^1=0$, equipped with a minimal $A_\infty$-structure
	extending its product. Assume that the Poincar\'e duality map
	$H\to H^\vee[-d]$ extends to an $\infty$-morphism of
	$H$-bimodules.
	If $d+i\in R^\times$ for $1\leqslant i\leqslant d-4$, then
	\[
	H\text{ is gauge formal}
	\quad\Longleftrightarrow\quad
	H/H^d\text{ is gauge formal},
	\]
	where the quotient carries the induced $A_\infty$-structure.
\end{theorem}

\begin{proof}
	Since $H^1=0$, higher operations and higher morphism components
	vanish on inputs in $H^d$. Thus the quotient structure is well
	defined, and the forward implication follows from
	Corollary~\ref{gauge formality}: every formality isotopy descends
	to $H/H^d$.
	Conversely, we assume that $\bar H:=H/H^d$ is gauge formal.
	We lift a formality isotopy through $H=H^{<d}\oplus H^d$ and
	transport the structure and the bimodule duality. We may therefore
	assume that $q\colon H\to\bar H$ and $\bar H$ are strict, so all
	higher operations on $H$ take values in $H^d$.
	
	\noindent The bimodule duality sends the unit Hochschild $0$-cochain to a closed functional $\tau$ whose length-zero component is the Poincar\'e trace $\varepsilon$, see \cite[Definition~9]{Tra08}.
	We use normalized deformation complexes throughout and write $C_*^+(H)$ for the normalized Hochschild chains with
	all entries in $H^{>0}$. Lemma~\ref{compo} and the Lie derivative identities give a
	filtered chain map
	\[
	\begin{aligned}
		F\colon\mathfrak g_H^{\varphi_t}
		&\longrightarrow
		(\mathfrak g_{\bar H}^{H},d^q)
		\oplus
		\bigl(\mathcal F^1(C_*^+(H)^\vee[-d]),d^{\varphi_t}\bigr),
		\\
		c&\longmapsto\bigl(q\star c,L_c^\vee\tau\bigr).
	\end{aligned}
	\]
	The dual is completed and filtered by bar length, and
	$\mathcal F^1$ consists of functionals vanishing in length zero.
	The Lie derivative $L_c$ acts by cyclic insertion on full chains
	before restriction to $C_*^+(H)$. We verify the hypotheses of Theorem~\ref{thm formality transmission}. Let us set $\Phi=\mathcal D_1(\varphi_t)$ and
	$\tau(\hbar)=\mathcal D_0(\tau)$.
	Since $q$ is strict and $d^\Phi\tau(\hbar)=0$, the same
	differentiation argument as in
	Lemma~\ref{piquant prop}\,\textup{(1)} gives
	\[
	\mathcal D_0(F)(\partial_\hbar\Phi)
	=
	\bigl(0,-d^\Phi(\partial_\hbar\tau(\hbar))\bigr).
	\]
	The primitive has positive bar length, since the length-zero
	component of $\tau(\hbar)$ is the constant $\varepsilon$.
	Thus $F_*(K_{\varphi_t}^n)=0$ for every $n\geqslant1$. We now consider \[[c]\in\ker E_2^{k,-1}(F) , \] with
	$k\geqslant2$.
	On this page, the trace component is the restriction of the
	functional associated by Poincar\'e duality to $\Delta c$, where
	$\Delta$ is the Hochschild BV operator for the strict product
	on $H$ \cite[Sections~2--3]{Tra08}.
	Lifting a primitive for $q\rhd c$ through $q$ gives a top-degree
	representative of $[c]$. We likewise extend a primitive for the
	trace component by zero on the distinguished unit input and
	dualize. Subtracting the resulting coboundary from $\Delta c$
	gives a representative of $\Delta[c]$ with values in
	$\left(H^{>0}\right)^\perp=H^d$. In our cohomological convention, the Euler endomorphism is
	$e_H(a)=|a|a$. Since $[c]$ and $\Delta[c]$ have top-degree
	representatives, \[[e_H]\smile[c]=[e_H]\smile\Delta[c]=0 \ .\]
	The BV identity, with our Gerstenhaber bracket convention, and
	Remark~\ref{rem:properadic-euler-identity} therefore give
	\[
	(1-k)[c]=(\Delta e_H)\smile[c]=d[c],
	\qquad (d+k-1)[c]=0.
	\]
	Here $\Delta e_H=d\,1$, since
	$\varepsilon(e_H(a))=d\,\varepsilon(a)$.
	Hence $E_2^{k,-1}(F)$ is injective whenever $d+k-1$ is a unit. 	This injectivity persists on later pages because every target
	$E_2$-class in deformation degree zero has a closed lift.
	Indeed, a mixed cochain of weight $p\geqslant1$ and degree zero
	vanishes on top-degree inputs: its output would have degree
	at least $d+p$, whereas $\bar H^{\geqslant d}=0$.
	Similarly, a positive-length trace functional of degree zero
	tests degree $d$, while a positive-length chain containing a
	top-degree entry has degree at least $d+1$.
	Since the trace functionals vanish in length zero, higher
	operations contribute nothing in either summand.
	Consequently, no higher differential enters the target
	obstruction line, and injectivity persists by induction. If $d+1,\ldots,d+n$ are units, the maps $E_k^{k,-1}(F)$ are
	therefore injective for $2\leqslant k\leqslant n+1$.
	Moreover, $\prod_{i=1}^n(d+i)=n!\binom{d+n}{n}$ shows that
	$n!$ is a unit. Theorem~\ref{thm formality transmission}\,\textup{(2)}
	therefore gives gauge $n$-formality.
	
	\noindent Finally, an operation of arity $r\geqslant3$ has output degree
	at least $r+2$ on positive inputs, so it vanishes for $r>d-2$.
	The case $d\leqslant4$ is therefore immediate; otherwise, gauge
	$(d-4)$-formality gives gauge formality, proving the converse.
\end{proof}

\begin{remark}[Partial formality]
	The proof also shows that, if $H/H^d$ is gauge formal, then
	$H$ is gauge $n$-formal whenever $d+1,\ldots,d+n$ are units.
	This partial conclusion applies to the geometric corollaries
	below as well, with $d$ replaced by $2d$ in the \'etale case.
\end{remark}

\begin{corollary}[Formality under puncturing]
	\label{formality under puncturing}
	Let $M^d$ be a closed connected $R$-oriented manifold,
	$d\geqslant2$, with $H^*(M;R)$ finite projective and
	$H^1(M;R)=0$. Fix $x\in M$.
	If $d+i\in R^\times$ for $1\leqslant i\leqslant d-4$, then
	\[
	C^*(M;R)\text{ is gauge formal}
	\quad\Longleftrightarrow\quad
	C^*(M\setminus\{x\};R)\text{ is gauge formal}.
	\]
\end{corollary}

\begin{proof}
	Excision and orientation identify restriction in cohomology with
	the quotient by $H^d(M;R)$.
	Transfer and minimal-model comparison \cite{Pet20} identify the
	induced quotient of a minimal model of $M$ with a minimal model
	of its puncture: the higher components of restriction vanish
	on top-degree inputs by the degree argument above. It remains to supply the bimodule duality.
	Cap product with the fundamental class gives an equivalence of
	$E_\infty$-modules \cite[Corollary~5.26]{GTZ12}.
	Restriction to $A_\infty$-bimodules and transfer therefore give
	the required lift of Poincar\'e duality and
	Theorem~\ref{thm:pd-transmission} applies.  
\end{proof}

\begin{remark}[Poincar\'e duality spaces]
	The same conclusion holds for a finite connected $R$-oriented
	Poincar\'e duality complex $X=Y\cup_\alpha e^d$, with
	$\dim Y\leqslant d-1$, replacing $(M,M\setminus\{x\})$ by $(X,Y)$
	and imposing the same cohomological and numerical assumptions
	on $X$.
	The cohomology sequence of the pair identifies restriction with
	the quotient by the top degree, and cap product supplies the
	same bimodule duality.
\end{remark}

\begin{remark}[Highly connected manifolds]
	If $M$ is $k$-connected, with $k\geqslant1$, the positive cohomology
	of its puncture lies in degrees $k+1,\ldots,d-k-1$.
	An operation of arity $r\geqslant3$ has output degree at least
	$rk+2\geqslant3k+2$, so the puncture is gauge formal when
	$d\leqslant4k+2$.
	Over $\mathbb Q$, the corollary thus recovers Miller's theorem
	\cite{Mil79}.
	The direct argument of \cite[Theorem~3.3]{CE24b}, with $\ell=3$,
	gives the same dimension bound over fields of characteristic
	different from $2$ and $3$. The numerical condition above is
	therefore not optimal in general.
\end{remark}

\noindent
For the \'etale analogue, let $K$ be separably closed and let $R$
be a finite commutative ring of order invertible in $K$, or
$\mathbb Z_\ell$ or a finite extension of $\mathbb Q_\ell$,
with $\ell$ invertible in $K$.
Write
$C^*_{\mathrm{\acute et}}(X;R)=R\Gamma(X_{\mathrm{\acute et}},R)$,
using continuous cochains for $\ell$-adic coefficients.

\begin{corollary}[\'Etale formality under puncturing]
	\label{cor:etale-puncturing}
	Let $X$ be a smooth proper geometrically connected $K$-variety
	of dimension $d>0$, with $H^*_{\mathrm{\acute et}}(X;R)$
	finite projective and $H^1_{\mathrm{\acute et}}(X;R)=0$.
	Fix $x\in X(K)$.
	If $2d+i\in R^\times$ for $1\leqslant i\leqslant2d-4$, then
	\[
	C^*_{\mathrm{\acute et}}(X;R)\text{ is gauge formal}
	\quad\Longleftrightarrow\quad
	C^*_{\mathrm{\acute et}}(X\setminus\{x\};R)
	\text{ is gauge formal}.
	\]
\end{corollary}

\begin{proof}
	Purity and localization identify restriction in cohomology with
	the quotient by the top degree: the Gysin class of $x$ has
	trace $1$ and generates $H^{2d}_{\mathrm{\acute et}}(X;R)$
	\cite[Sections~16 and~24]{Mil13}.
	Thus transfer identifies the quotient of a minimal model of $X$
	with a minimal model of its puncture. After trivializing the Tate line $R(-d)$, multiplication and
	trace give Poincar\'e duality in formal dimension $2d$.
	The cochain complex is perfect because its cohomology is bounded
	and finite projective, so this duality is an equivalence of
	$E_\infty$-modules.
	Restriction to bimodules and transfer provide the duality required
	by Theorem~\ref{thm:pd-transmission}.
	The result follows as in the topological case.
\end{proof}

\subsection{Intrinsic formality}

A cohomology algebra $H$ is called \emph{intrinsically formal}
if every algebra whose cohomology algebra is isomorphic to $H$
is gauge formal. Although intrinsic formality is stronger than
formality, it can sometimes be easier to establish, as illustrated
by Tamarkin's proof of Kontsevich's formality theorem; see
\cite{Hinich03bis}. Introduced by Sullivan
and developed through the obstruction theory of Halperin--Stasheff, this notion admits several classical criteria. In particular, quotients of finitely generated free graded-commutative algebras by homogeneous
regular sequences are intrinsically formal \cite{Sul77,HJ79,FH82,Lup90};
see \cite[Section~3.2]{Suc23} for an overview and
\cite{MP91} for a deformation-theoretic treatment. More recently, a formulation in terms of
Kaledin classes for properadic algebraic structures may be found in \cite[Theorem~4.6]{Kaledin}. 

\medskip

\noindent The formality transmission yields an intrinsic formality criterion, extending former criteria to arbitrary coefficient rings. Throughout this section, suppose that the ring $R$ and the coproperad $\C$ are such that any $\Omega \C$-algebra admits a transferred structure. 

\begin{theorem}[Intrinsic formality]\label{intrinsic}  Let $(H, \psi)$ be a graded $R$-module equipped with a strict   $\Omega \mathcal{C}$-algebra structure $\psi = \psi^{(1)}$. Let  $E \coloneqq E(\C, H, \psi)$ be its properadic spectral sequence. 
	\begin{enumerate}
		\item If $E^{k,-1}_2 = 0$ for all $k >1$, then $(H, \psi)$ is intrinsically formal, i.e. any $\Cobar \C$-algebra whose induced structure is $(H, \psi)$ is gauge formal. 
		\item Let $n \geqslant 1$. If $E^{k,-1}_2 = 0$ for  \(1 <  k\leqslant n + 1 \), then $(H, \psi)$ is intrinsically $n$-formal, i.e. any $\Cobar \C$-algebra whose induced structure is $(H, \psi)$ is gauge $n$-formal.  
	\end{enumerate}
\end{theorem}

\begin{proof} Let $(A, \varphi)$ be an $\Omega \C$-algebra whose induced structure is $ (H , \psi)$. If $E^{k,-1}_2 = 0$, we have  \[E^{k,-1}_k (\C, H(A), \varphi_t) = 0 \ , \]  for all $k >1$. Over a $\mathbb{Q}$-algebra, point (1) then follows from Corollary \ref{appli B} (1): by considering $B = 0$ the involved maps are necessarily injective. 	Similarly, if $n !$ is a unit in $R$, point (2) follows from Corollary \ref{appli B} (2). Nonetheless, in this setting formality transmission can be refined to avoid these characteristic constraints. Let us recall \cite[Proposition~2.33]{CE24b}. Suppose that $ \varphi - \psi \in \mathcal{F}^{n} \mathfrak{h}$, for $n \geqslant 1$.  If the homology class \[\vartheta_{n} \coloneqq \left[\varphi^{(n)} \right] \in H_{-1}\left( \mathfrak{h} / \mathcal{F}^{n+1}\mathfrak{h}  \right) \quad \mbox{and} \quad \mathfrak{h} \coloneqq \mathfrak{g}_H^{\psi} \]  vanishes, the $\Cobar \C$-algebra structures $\varphi$ and $\psi$ are gauge $(n-1)$-homotopy equivalent. More precisely, if $\varphi^{(n)} = d^{\psi} (\nu)$, the $\infty$-isotopy $f = 1 +  \nu^{(n-1)}$ and the Maurer--Cartan element $\phi = f \cdot \varphi$ are such that $\phi - \psi \in \mathcal{F}^{n +1} \mathfrak{h}$. We apply this construction to $\phi_2:=\varphi_t$.
	For $k\geqslant 2$, the group $E_2^{k,-1}$ is the homology
	at the middle term of
	\[
	\mathfrak h_0^{(k-1)}
	\xrightarrow{\, d^\psi \,}
	\mathfrak h_{-1}^{(k)}
	\xrightarrow{\, d^\psi \, }
	\mathfrak h_{-2}^{(k+1)}.
	\]
	Suppose inductively that
	$\phi_k-\psi\in\mathcal F^k\mathfrak h$. Whenever $E_2^{k,-1}=0$,
	there exists $\nu^{(k-1)}\in\mathfrak h_0^{(k-1)}$ such that
	\[
	\phi_k^{(k)}=d^\psi\left(\nu^{(k-1)}\right).
	\]
	Thus the obstruction class vanishes, and the construction
	above yields
	\[
	f_k:=1+\nu^{(k-1)},
	\qquad
	\phi_{k+1}:=f_k\cdot\phi_k,
	\qquad
	\phi_{k+1}-\psi\in\mathcal F^{k+1}\mathfrak h.
	\]
	This induction involves no division. Under the hypothesis of \textup{(1)}, the construction
	continues indefinitely. Since $f_k-1$ has weight $k-1$,
	the successive compositions stabilize in each weight,
	defining an $\infty$-isotopy
	$(H,\varphi_t)\rightsquigarrow(H,\psi)$.
	Thus $(A,\varphi)$ is gauge formal. Under the hypothesis of \textup{(2)}, iterating for
	$k=2,\ldots,n+1$ yields an $\infty$-isotopic structure
	$\phi_{n+2}$ satisfying
	\[
	\phi_{n+2}-\psi\in\mathcal F^{n+2}\mathfrak h,
	\]
	which proves gauge $n$-formality.
	Since $(A,\varphi)$ was arbitrary, both conclusions
	are intrinsic.
\end{proof}

\begin{example}[Properadic coformality of spheres]
	An application of this framework is the properadic coformality of spheres. Recall that a topological space $X$ is coformal if the chain algebra of its based loop space, $C_*(\Omega X;\mathbb{Q})$, is formal as an $\mathcal{A}_{\infty}$-algebra. When $X$ is an oriented manifold, this algebra carries additional structure: Kontsevich, Takeda, and Vlassopoulos showed in \cite{KTV21} that it admits a pre-Calabi--Yau structure encoding Poincaré duality. Properadic coformality strengthens ordinary coformality by requiring the zig-zag of quasi-isomorphisms realizing formality to preserve this additional structure, and hence Poincaré duality. In \cite{ET24}, we apply intrinsic formality criteria to show that $C_*(\Omega S^n;\mathbb{Q})$ is intrinsically formal as a pre-Calabi--Yau algebra for every $n \geqslant 1$.
\end{example}

\subsection{Formality descent}

Let $S$ be a flat commutative $R$-algebra. If the $\Cobar \C$-algebra $(A, \varphi)$ is formal, then the algebra $(A \otimes_{R} S,\varphi \otimes 1)$ stays formal. Formality descent gives sufficient conditions under which the converse holds true.
For extensions of characteristic-zero fields, formality descent holds for finite-type structures \cite{Sul77, NM78, HJ79, FOT08}. This was extended to dg operads in \cite{GNPR05} and to any $\Omega \C$-algebras over a commutative ring in \cite{Kaledin}. This latter generalization can be seen as a consequence of formality transmission. This approach significantly simplifies the proof but does not offer any meaningful generalization. We give a counterexample showing that the characteristic-zero
assumption cannot be omitted in general: full gauge formality
need not descend along faithfully flat extensions in positive
characteristic, even for finite-dimensional strictly unital
$A_\infty$-algebras.

\begin{theorem}[Formality descent, {\cite[Theorem~4.1]{Kaledin}}]\label{formality descent}
	Let $S$ be a faithfully flat commutative $R$-algebra and denote $A_{S} \coloneqq A \otimes_{R} S$ Assume that scalar extension induces isomorphisms
	\[
	(\mathfrak g_{H(A)})^{(w)}\otimes_R S
	\xrightarrow{\sim}
	(\mathfrak g_{H(A) \otimes S})^{(w)}
	\qquad\text{for every }w\geq1 . 
	\]
	\begin{enumerate}
		\item If $R$ is a $\mathbb{Q}$-algebra, the $\Cobar \C$-algebra $(A, \varphi)$ is gauge formal if and only if $(A_{S},\varphi \otimes 1)$ is gauge formal.
		\item Let $n \geqslant 1$ be an integer such that $n !$ is a unit in $R$.  The $\Cobar \C$-algebra $(A, \varphi)$ is gauge $n$-formal if and only if $(A_{S},\varphi \otimes 1)$ is gauge $n$-formal.
	\end{enumerate}
\end{theorem}

\begin{proof}
	Write $H:=H(A)$ and $\mu:=\varphi_t$. Flatness identifies
	$H(A_S)$ with $H_S:=H\otimes_R S$, with transferred structure
	$\mu_S:=\mu\otimes 1$. Scalar extension induces a filtered map
	\[
	\iota:\mathfrak g_H^\mu\longrightarrow
	\mathfrak g_{H_S}^{\mu_S},
	\]
	where the target is formed over $S$ and regarded as an
	$R$-complex. The assumed weightwise base-change isomorphisms identify
	the associated graded complexes after scalar extension.
	Since $S$ is flat over $R$, taking homology commutes with
	tensoring by $S$. Induction on $r$ therefore gives
	\[
	E_r^{p,q}(\mathfrak g_{H_S}^{\mu_S})
	\cong
	E_r^{p,q}(\mathfrak g_H^\mu)\otimes_R S
	\]
	for every finite page $r\geqslant0$ and every bidegree $(p,q)$.
	Under these identifications, the maps induced by $\iota$
	are the canonical maps $x\mapsto x\otimes1$, which are
	injective by faithful flatness. Moreover, naturality gives
	\[
	\iota_*(K_\mu)=K_{\mu_S},
	\qquad
	\iota_*(K_\mu^n)=K_{\mu_S}^n.
	\]
	If $A_S$ is gauge formal, respectively gauge $n$-formal,
	the corresponding class vanishes by Theorem \ref{B prop}. 
	Theorem \ref{thm formality transmission} therefore proves both descent implications under the stated assumptions on $R$ and $n$.
	The converse implications follow by extending the
	$\infty$-isotopies of Corollary \ref{gauge formality} by scalars.
\end{proof}

\begin{remark}
	The weightwise base-change hypothesis holds in the
	following situations.
	\begin{enumerate}
		\item[$\centerdot$] $S$ is finite projective as an $R$-module.
		This includes finite extensions of fields and requires
		no finiteness assumptions on $C$ or $H$.
		
		\item[$\centerdot$] $C$ is a cooperad and, for every $w\geq1$, the following total $R$-module is finitely presented
		\[
		C^{(w)}(H)
		:=
		\bigoplus_{n\geq0}
		C^{(w)}(n)\otimes_{R[\mathbb S_n]}H^{\otimes n} \ .
		\]

		\item[$\centerdot$] $H$ is bounded, degreewise finitely generated
		projective, each $C^{(w)}(m,n)$ is degreewise finitely
		presented over $R$, and only finitely many pairs $(m,n)$
		occur in each weight.
	\end{enumerate}
\end{remark}

\begin{proposition}[Failure of formality descent in characteristic two]\label{failure}
	Let
	\[
	R=\mathbb F_2(t),
	\qquad
	S=R[u]/(u^2-t).
	\]
	There exists an $A_\infty$-algebra $(A,\varphi)$ over $R$
	that is not gauge formal but becomes so over $S$.
\end{proposition}

\begin{proof}
Let us use the cohomological convention and consider the graded vector space
	\[
	A=R1\oplus Ra\oplus Rb\oplus Rc,
	\qquad
	|1|=0,\quad |a|=|b|=1,\quad |c|=2 \ ,
	\] equipped with the strictly unital minimal $A_\infty$-structure
	$\varphi$ defined as follows. The binary operation has unit $1$
	and satisfies $\varphi_2(a,a)=\varphi_2(b,b)=c,$
	all other products between $a,b,c$ being zero. The only
	nonzero higher operation is
	\[
	\varphi_4(b,b,b,b)=tc,
	\]
	with $\varphi_4$ zero on all other basis tensors. The scalar extension $A_S:=A\otimes_R S$ is gauge formal.
	Indeed, we define an
	$\infty$-isotopy
	\[
	f:(A_S,\varphi\otimes1)
	\rightsquigarrow
	(A_S,\varphi_2\otimes1), \qquad 
	f_1=\mathrm{id} ,
	\quad
	f_2(b,b)=ub,
	\quad
	f_r=0\quad(r\geqslant3),
	\]
	with $f_2$ zero otherwise. Suppose, however, that $(A,\varphi)$ is gauge formal and let 
	\[
	g:(A,\varphi)\rightsquigarrow(A,\varphi_2) \ ,
	\] be an $\infty$-isotopy. Let us set $x=a+b$ and, by degree, write $g_2(x,x)=ra+sb$ for some $r,s\in R$.
	Since $\varphi_2(x,x)=0$, the  $\infty$-morphism identity gives
	\[
	tc=\varphi_2\bigl(g_2(x,x),g_2(x,x)\bigr)
	=(r+s)^2c,
	\]
	the terms involving $g_3$ cancelling by commutativity
	in characteristic $2$. Thus $t=(r+s)^2$, contradicting
	the fact that $t$ is not a square in $\mathbb F_2(t)$.
	Hence $(A,\varphi)$ is not gauge formal.
\end{proof}

\subsection{Formality in families}
Kaledin's approach to formality in families \cite{Kal07}
was extended by Lunts \cite{Lun07} to $A_\infty$-algebras
over a $\mathbb Q$-algebra, under suitable finiteness and
projectivity assumptions. These results were generalized
in \cite[Section 4.2]{Kaledin} to algebras over colored
operads. As with formality descent,
formality in families can also be deduced from formality
transmission. We collect the resulting criteria in the
following theorem. Throughout this subsection, suppose
that $\C$ is a cooperad.

\begin{theorem}[Formality in families]\label{formality in families}
	\label{thm:formality-families}
	For a commutative $R$-algebra $S$, set $H:=H(A)$,
	\[
	M_k(S):=E_2^{k,-1}(\C_S,H_S,\varphi_*\otimes1),
	\qquad M_k:=M_k(R).
	\]
	Assume that $\C^{(w)}(H)$ is finitely presented as a total
	$R$-module for every $w\geqslant0$.	Let $n\geqslant1$ with $n!\in R^\times$. All conditions on $M_k$ below concern
	$2\leqslant k\leqslant n+1$.
	For statements \textup{(2)--(4)}, assume additionally that
	each $M_k$ is projective and that the canonical maps
	\begin{equation}
		M_k\otimes_R S\longrightarrow M_k(S)
		\tag{$\ast$}\label{eq:families-bc}
	\end{equation}
	are isomorphisms for every commutative $R$-algebra $S$.
	Then the following hold.
	\begin{enumerate}
		\item \emph{Generic fibre.}		Suppose that $R$ is an integral domain and each $M_k$
		is torsion-free. Then $A$ is gauge $n$-formal if and
		only if $A_{\operatorname{Frac}(R)}$ is gauge $n$-formal,
		or equivalently, if and only if every fibre
		$A_{\kappa(\mathfrak p)}$ is gauge $n$-formal.
		
		\item \emph{Specialization.}
		The locus
		\[
		F_n(A):=
		\{\mathfrak p\in\operatorname{Spec}R
		\mid A_{\kappa(\mathfrak p)}
		\text{ is gauge }n\text{-formal}\}
		\]
		is closed under specialization.
		
		\item \emph{Infinitesimal thickenings.}
		Let $I\subseteq R$ be an ideal such that
		$\bigcap_{\ell\geqslant1}I^\ell=0$.
		Then $A$ is gauge $n$-formal if and only if
		$A_{R/I^\ell}$ is gauge $n$-formal for every
		$\ell\geqslant1$.
		
		\item \emph{Closed fibres.}
		Suppose that $\operatorname{Jac}(R)=0$.
		Then $A$ is gauge $n$-formal if and only if
		$A_{R/\mathfrak m}$ is gauge $n$-formal for every
		maximal ideal $\mathfrak m$.
	\end{enumerate}
	
	\noindent If $R$ is a $\mathbb Q$-algebra, the same statements hold
	for gauge formality, with the corresponding conditions
	imposed for every $k\geqslant2$.
\end{theorem}

\begin{proof}
	Let $(S_\lambda)_\lambda$ be a family of $R$-algebras such
	that every $A_{S_\lambda}$ is gauge $n$-formal, and consider
	\[
	f:(H,\varphi_t)\longrightarrow
	\prod_\lambda(H\otimes_R S_\lambda,\varphi_t\otimes1).
	\]
	Its target is gauge $n$-formal componentwise.
	Scalar-extension adjunction identifies the spectral sequence
	of $f$ with the product of those of the fibres, which are
	stationary from $E_2$ through $E_{n+1}$.
	Consequently, injectivity of the induced map
	\begin{equation}
		M_k\longrightarrow\prod_\lambda M_k(S_\lambda)
		\label{eq:joint-detection}
	\end{equation}
	on $E_2^{k,-1}$ persists through $E_k$.
	Thus Corollary~\ref{appli B} \textup{(2)} applies whenever
	\eqref{eq:joint-detection} is injective for
	$2\leqslant k\leqslant n+1$.
	
	\smallskip
	\noindent
	\textup{(1)}
	For $S=\operatorname{Frac}(R)$, flat base change identifies
	\eqref{eq:joint-detection} with
	\[
	M_k\longrightarrow M_k\otimes_R\operatorname{Frac}(R),
	\]
	which is injective by torsion-freeness.
	The equivalence with formality of all fibres follows by
	scalar extension and inclusion of the generic fibre.
	
	\smallskip
	\noindent
	\textup{(2)}
	Let $\mathfrak p\subseteq\mathfrak q$ and suppose that
	$A_{\kappa(\mathfrak p)}$ is gauge $n$-formal.
	For $T=R/\mathfrak p$, condition~\eqref{eq:families-bc}
	identifies $M_k(T)$ with the projective $T$-module
	$M_k\otimes_R T$, hence it is torsion-free.
	Applying \textup{(1)} over $T$ gives gauge $n$-formality
	of $A_T$, and scalar extension gives that of
	$A_{\kappa(\mathfrak q)}$.
	
	\smallskip
	\noindent
	\textup{(3)--(4)}
	Take respectively $S_\ell=R/I^\ell$ and
	$S_{\mathfrak m}=R/\mathfrak m$.
	By \eqref{eq:families-bc}, the corresponding kernels are
	\[
	\bigcap_{\ell\geqslant1}I^\ell M_k=0,
	\qquad
	\bigcap_{\mathfrak m\in\operatorname{MaxSpec}R}
	\mathfrak mM_k=0.
	\]
	These vanish because $M_k$ is a direct summand of a free
	module and, respectively, $\bigcap_\ell I^\ell=0$ and
	$\operatorname{Jac}(R)=0$. The converse implications follow by scalar extension.
	Over a $\mathbb Q$-algebra, the same argument in all weights
	uses Corollary~\ref{appli B} \textup{(1)}.
\end{proof}

\noindent As an arithmetic application, we recover the following
result of El Haouari.

\begin{corollary}[El Haouari~\cite{El92}]
	\label{cor:el-haouari}
	Let $X$ be a finite simply connected CW-complex. Then
	$C^*(X;\mathbb Q)$ is formal as an associative dg algebra if and
	only if $C^*(X;\mathbb F_p)$ is formal for all but finitely many
	primes $p$.
\end{corollary}

\begin{proof}
Let us choose a finite simplicial cochain model $A$ of $X$ over
	$\mathbb Z$ and fix $n>\dim X$. After passing to
	$R=\mathbb Z[1/N]$, its cohomology $H$ is finite free, and a
	unit-preserving contraction gives a strictly unital minimal model
	compatible with scalar extension. Since $H^{>0}$ starts in degree
	two, an $r$-ary higher operation on positive inputs has output
	degree at least $r+2$. Restricting to $H^{>0}$ and adjoining the
	unit therefore identifies gauge $n$-formality with gauge formality.  Enlarging $N$, we may invert $n!$ and split the finitely many
	finite free complexes computing $M_2,\ldots,M_{n+1}$ for this
	associative model. Thus the $M_k$ are finite free and satisfy
	\eqref{eq:families-bc}. Since
	$\operatorname{Frac}(R)=\mathbb Q$ and $\operatorname{Jac}(R)=0$,
	Theorem~\ref{thm:formality-families} \textup{(1), (4)} identifies
	gauge $n$-formality over $\mathbb Q$ with gauge $n$-formality over
	every $\mathbb F_p$ with $p\nmid N$.
	Over fields, gauge formality agrees with dg formality.
	The result follows, enlarging $N$ to include the exceptional
	primes in the converse direction.
\end{proof}

\begin{remark}
	The same equivalence holds for any unital dg associative algebra
	over $\mathbb Z[1/N]$ whose underlying complex is bounded and
	finite projective, provided that its rational cohomology is
	$\mathbb Q\cdot1\oplus V$ with $V^i=0$ for $i<2$.
	Here only primes $p\nmid N$ are considered.
\end{remark}

\subsection{Lifts of homology automorphism}

Motivated by the Weil conjectures, there is a long-standing tradition of using \emph{purity} as a tool to establish formality. We refer to \cite{Notes} for a survey of this approach.

\begin{theorem}[Purity]
Assume that \(C^{(p)}\) is concentrated in homological degree \(p\) for every \(p\geqslant 0\). Let $\alpha$ be a unit in $R$ such that $\alpha^k-1$ is a unit  for every positive integer $k$. Let $\sigma_{\alpha}$ be the grading automorphism, acting on homology by multiplication by $\alpha^k$ in homological degree $k$. If $(A,\phi)$ satisfies a purity property, i.e. if $\sigma_{\alpha}$ admits a lift to the chain level, then the algebra is formal.
\end{theorem}

\noindent This criterion was proved by Sullivan \cite{Sul77} for associative algebras, and was later generalized to many other types of algebraic structures; see \cite{GNPR05, Petersen14,Dup16,CH20,DCH21,CS22, CH22}.  In \cite[Theorem~4.9]{Kaledin}, this purity criterion is generalized using Kaledin classes to any algebra over a properad and in particular to other homology automorphisms. This generalization can be seen as a consequence of formality transmission. This approach allows us to refine the hypothesis on the automorphism and yields the following result.

\begin{theorem}[Automorphism lifts]\label{relevements}
 Let $u \in \mathrm{Aut}(H(A), \varphi_*)$ be a strict automorphism of the homology that admits a chain level lift $v$.   
Let us define $S \coloneqq  p_{\infty} \circledcirc  v \circledcirc i_{\infty} $ and consider for every $k \geqslant 1$, the map induced by $\mathrm{Ad}_S - \mathrm{id}$ on the $E_k$ page of the properadic spectral sequences \[ \theta_{k} :  E(\C, H(A), \varphi_t)^{k,-1}_k  \longrightarrow E(\C,  H(A), \varphi_t)^{k,-1}_k      \ . \] 
\begin{enumerate}
	\item If $R$ is a $\mathbb{Q}$-algebra, and   $\theta_k$ is injective for every $k > 1$, then $(A, \varphi)$ is gauge formal.
	\item Let $n \geqslant 1$ be an integer such that $n !$ is a unit in $R$.  If \(\theta_k\) is
	injective for \(1  <  k\leqslant n + 1 \), then $(A, \varphi)$ is gauge $n$-formal. 
\end{enumerate}
\end{theorem}

\begin{proof}
This is a direct consequence of Corollary \ref{appli A}. 
\end{proof}

\noindent This result generalizes the criterion of \cite[Theorem~4.10]{Kaledin}. Indeed, the invertibility of $\mathrm{Ad}_u-\mathrm{id}$ implies, in particular, the injectivity of the maps $\theta_k$. More importantly, under this stronger assumption, the additional hypotheses that $R$ be a $\mathbb{Q}$-algebra, or that there exist an integer $n\geqslant 1$ such that $n!$ is a unit in $R$, are no longer needed. The following theorem makes this refinement precise and may be viewed as a sharper form of the formality transmission phenomenon in this setting.

\begin{theorem}\label{rev2}
	Let $u$ be an automorphism of $(H(A),\varphi_*)$
	admitting a chain-level lift $v$. Set
	\[
	s:=p_\infty\circledcirc v\circledcirc i_\infty,
	\qquad
	E:=E(C,H(A),\varphi_t), \qquad
	T_1:=E_1(\operatorname{Ad}_s)-\mathrm{id}
	\]
	be the induced endomorphism of the first page.
	Then the following statements hold.
	\begin{enumerate}
		\item Let $n\geqslant 1$. If the maps
		\[
		T_1:E_1^{k,0}\longrightarrow E_1^{k,0},
		\qquad
		T_1:E_1^{k+1,-1}\longrightarrow E_1^{k+1,-1}
		\]
		are isomorphisms for every $1\leqslant k\leqslant n$, then
		$(A,\varphi)$ is gauge $n$-formal.
		
		\item If these maps are isomorphisms for every $k\geqslant1$,
		then $(A,\varphi)$ is gauge formal.
	\end{enumerate}
	No additional assumption on the characteristic of $R$
	is required.
\end{theorem}

\begin{proof}
	Set $H:=H(A)$ and $\mathfrak g:=\mathfrak g_H$.
	Since $E_1^{p,q}\cong\mathfrak g_q^{(p)}$, the hypotheses
	apply to $T:=\operatorname{Ad}_u-\mathrm{id}$ in the
	corresponding weights and degrees. Starting with $\psi_1:=\varphi_t$ and $s_1:=s$, we construct
	successively a structure $\psi_k$ and an
	$\infty$-automorphism $s_k$ of $\psi_k$ such that
	\[
	\psi_k\equiv\varphi_*\pmod{F^{k+1}\mathfrak g},
	\qquad
	s_k\equiv u\pmod{F^k\mathfrak g}.
	\]
	Suppose these conditions hold. Define
	\[
	a_k:=(s_k\circledcirc u^{-1})^{(k)},
	\qquad
	\xi_k:=T^{-1}(a_k)\in\mathfrak g_0^{(k)},
	\qquad
	f_k:=1+\xi_k,
	\]
	and set
	\[
	\psi_{k+1}:=f_k\cdot\psi_k,
	\qquad
	s_{k+1}:=f_k\circledcirc s_k\circledcirc f_k^{-1}.
	\]
	The weight-$k$ conjugation formula gives
	\[
	(s_{k+1}\circledcirc u^{-1})^{(k)}
	=a_k+\xi_k-\operatorname{Ad}_u(\xi_k)=0.
	\]
	Thus $s_{k+1}\equiv u\pmod{F^{k+1}\mathfrak g}$.
	Moreover, since $f_k-1$ has weight $k$, we still have
	$\psi_{k+1}\equiv\varphi_*\pmod{F^{k+1}\mathfrak g}$.
	The $\infty$-automorphism identity for $s_{k+1}$ in
	weight $k+1$ therefore reduces to
	\[
	T\left(\psi_{k+1}^{(k+1)}\right)=0.
	\]
	Injectivity yields
	$\psi_{k+1}^{(k+1)}=0$, completing the induction. For (1), performing the construction for $k=1,\ldots,n$
	gives an $\infty$-isotopy from $\varphi_t$ to a structure
	congruent to $\varphi_*$ modulo $F^{n+2}\mathfrak g$.
	Corollary~\ref{gauge formality} implies gauge $n$-formality. For (2), the compositions
	$f_N\circledcirc\cdots\circledcirc f_1$ stabilize in each
	weight and hence define an $\infty$-isotopy from
	$\varphi_t$ to $\varphi_*$. Gauge formality follows again
	from Corollary~\ref{gauge formality}.
\end{proof}

\begin{example}\label{exemple pureté}
A key example is given by the action of a Frobenius on the étale cohomology of any smooth projective variety, see \cite[Section 4.5]{Kaledin}. 
\end{example}

\subsection{Formality under change of algebraic structure}

Saleh proved that a commutative dg algebra over a field of
characteristic zero is formal if and only if its underlying
associative dg algebra is formal~\cite{Sal17}.
Campos--Petersen--Robert-Nicoud--Wierstra  \cite{CPW} extended this comparison
to other operads. The criterion below extends their characteristic-zero comparison
by replacing the existence of an infinitesimal
bimodule retraction with injectivity on the relevant
spectral-sequence terms, and yields finite-order comparisons
over more general coefficient rings.

\begin{theorem}[Change of structure]\label{cooperad change}
	Let $(A,\varphi)$ be an $\Omega\C$-algebra admitting a transferred
	structure $\varphi_t$, and let $\iota:\C'\to\C$ be a morphism of
	reduced weight-graded dg coproperads over $R$.
	Precomposition induces a filtered morphism of deformation
	complexes, and hence maps
	\[
	\theta_k:
	E_k(\C,H(A),\varphi_t)^{k,-1}
	\longrightarrow
	E_k(\C',H(A),\varphi_t\circ\iota)^{k,-1}.
	\]
	\begin{enumerate}
		\item If $R$ is a $\mathbb Q$-algebra and $\theta_k$ is injective
		for every $k\geqslant2$, then $(A,\varphi)$ is gauge formal if
		and only if $(A,\varphi\circ\iota)$ is gauge formal.
		\item Let $n\geqslant1$ with $n!\in R^\times$.
		If $\theta_k$ is injective for $2\leqslant k\leqslant n+1$,
		then $(A,\varphi)$ is gauge $n$-formal if and only if
		$(A,\varphi\circ\iota)$ is gauge $n$-formal.
	\end{enumerate}
\end{theorem}

\begin{proof}
	Restriction along $\iota$ preserves $\infty$-isotopies and
	commutes with the weight decomposition. It therefore preserves
	gauge formality and gauge $n$-formality, and carries the
	(truncated) Kaledin class of $\varphi_t$ to that of
	$\varphi_t\circ\iota$.
	The converse follows from Theorem~\ref{thm formality transmission}.
\end{proof}

\begin{corollary}[Associative and commutative algebras]
	\label{change en char p}
	Let $A$ be a commutative dg algebra over $R$ admitting a
	transferred $C_\infty$-structure.
	\begin{enumerate}
		\item If $R$ is a $\mathbb Q$-algebra, then $A$ is gauge formal
		as a commutative dg algebra if and only if its underlying
		associative dg algebra is gauge formal.
		\item Let $n\geqslant1$ with $(n+1)!\in R^\times$.
		Then $A$ is gauge $n$-formal as a commutative dg algebra if and
		only if its underlying associative dg algebra is gauge $n$-formal.
	\end{enumerate}
\end{corollary}

\begin{proof}
	Let $I:\mathfrak h\hookrightarrow\mathfrak g$ be the
	Harrison--Hochschild inclusion for the transferred structure on
	$H(A)$, with weight equal to arity minus one.
	Over a $\mathbb Q$-algebra, the PBW decomposition gives a filtered
	chain retraction of $I$, see
	\cite[Proposition~2.10 and the proof of Proposition~3.10]{CPW}.
	Thus Theorem~\ref{cooperad change} (1) applies. For the second assertion, put $m=n+1$ and let $W_m\mathrm{Ass}$
	be the $m$th PBW filtration term: the span of products of at most
	$m$ Lie operations. Integral PBW and symmetrization give
	\[
	W_m\mathrm{Ass}\cong
	\bigoplus_{r=1}^{m}\mathrm{Sym}^{r}(\mathrm{Lie}),
	\]
	using only the denominators $r!$ for $r\leqslant m$.
	Lie substitutions and the adjoint action preserve these summands.
	Projection onto the first summand therefore induces a filtered
	chain retraction $P:\mathfrak g_m\to\mathfrak h$, where
	$\mathfrak g_m\subseteq\mathfrak g$ contains all Harrison cochains
	and all Hochschild cochains of arity at most $m$. For $2\leqslant k\leqslant m$, let $x$ represent a class in
	$\ker\theta_k$, and choose $y$ with $I(x)-dy\in F^{k+1}$.
	Since $d$ strictly increases weight, we may discard the components
	of $y$ of weights at least $k$. Then $y$ has arity at most
	$k\leqslant m$, and applying $P$ gives
	\[
	x-dP(y)\in F^{k+1}.
	\]
	Thus $[x]=0$ on $E_k$, proving the required injectivity and
	Theorem~\ref{cooperad change} (2) applies.
\end{proof}

\begin{corollary}[Zinbiel and dendriform algebras]
	\label{zinbiel en char p}
	Let $A$ be a non-unital dg Zinbiel algebra over $R$ admitting
	a transferred structure. Consider its underlying dendriform
	algebra.
	\begin{enumerate}
		\item If $R$ is a $\mathbb Q$-algebra, then $A$ is gauge formal
		as a Zinbiel algebra if and only if its underlying dendriform
		algebra is gauge formal.
		\item Let $n\geqslant1$ with $(n+1)!\in R^\times$.
		Then $A$ is gauge $n$-formal as a Zinbiel algebra if and only if
		its underlying dendriform algebra is gauge $n$-formal.
	\end{enumerate}
\end{corollary}

\begin{proof}
	The morphism $\mathrm{Dend}\to\mathrm{Zinb}$ is Koszul dual to
	$\mathrm{Leib}\to\mathrm{Dias}$. Under the  identifications
	\[
	\mathrm{Leib}\cong\mathrm{Lie}\otimes\mathrm{Perm},
	\qquad
	\mathrm{Dias}\cong\mathrm{Ass}\otimes\mathrm{Perm},
	\]
	where $\otimes$ is the Hadamard product, the latter morphism is
	$(\mathrm{Lie}\to\mathrm{Ass})\otimes\id_{\mathrm{Perm}}$,
	see \cite{Cha01} and \cite[Section~6.2]{Gri14}.
	Tensoring the full or truncated PBW retraction from the preceding
	proof with $\id_{\mathrm{Perm}}$ gives a retraction of infinitesimal
	$\mathrm{Leib}$-bimodules, since compositions are computed
	factorwise.
	The induced retraction of deformation complexes satisfies the
	same spectral-sequence argument as in
	Corollary~\ref{change en char p}.
	The result follows from Theorem~\ref{cooperad change}.
\end{proof}
\bibliographystyle{alpha}
\bibliography{bib}

\end{document}